\documentclass[reqno]{amsart}

\usepackage[utf8x]{inputenc}

\usepackage{color}
\usepackage[dvipsnames]{xcolor}

\usepackage{ifpdf}
\ifpdf 
\usepackage[pdftex]{graphicx}   
\usepackage[pdftex,     
plainpages=false,   
breaklinks=true,    
colorlinks=true,
linkcolor=red,
citecolor=green,
pdftitle=My Document
pdfauthor=My Good Self
]{hyperref} 
\else 
\usepackage{graphicx}       
\usepackage{hyperref}       
\fi 

\usepackage{subfig}
\usepackage{glossaries}
\usepackage{glossary-mcols}

\usepackage{aurical}
\usepackage{amsfonts,amsmath}	
\usepackage{amssymb}
\usepackage{verbatim}
\usepackage{amsopn}
\usepackage[english]{babel}
\usepackage{amsthm}
\usepackage{enumerate}
\usepackage{mathrsfs}	
\usepackage{mathtools}
\usepackage{esint}
\usepackage{todonotes}

\usepackage{bm}
\usepackage{bbm}
\usepackage{caption}

\theoremstyle{definition} \newtheorem{definition}{Definition}[section]
\theoremstyle{definition} \newtheorem{remark}[definition]{Remark}
\theoremstyle{plain} \newtheorem{lemma}[definition]{Lemma}
\theoremstyle{plain} \newtheorem{proposition}[definition]{Proposition}
\theoremstyle{plain} \newtheorem{theorem}[definition]{Theorem}
\theoremstyle{plain} \newtheorem{corollary}[definition]{Corollary}
\theoremstyle{definition} 
\theoremstyle{plain} 
\theoremstyle{definition}

\DeclareMathOperator{\BV}{BV}

\DeclareMathOperator{\dist}{dist}
\DeclareMathOperator{\supp}{supp}

\DeclareMathOperator{\Lip}{Lip}

\newcommand{\E}{\mathcal{E}}
\newcommand{\R}{\mathbb{R}}
\newcommand{\Q}{\mathbb{Q}}
\newcommand{\N}{\mathbb{N}}
\newcommand{\Z}{\mathbb{Z}}

\newcommand{\TV}{\text{\rm Tot.Var.}}

\newcommand{\e}{\varepsilon}

\newcommand{\loc}{\text{\rm loc}}
\newcommand{\D}{\mathcal{D}}
\newcommand{\Id}{\mathrm{id}}

\newcommand{\M}{\mathcal{M}}

\renewcommand{\L}{\mathscr L}

\newcommand{\1}{\mathbbm 1}

\renewcommand{\S}{\mathbb S}

\renewcommand{\div}{\mathrm{div}}
\renewcommand{\L}{\mathscr L}
\renewcommand{\H}{\mathscr H}

\numberwithin{equation}{section} 

\theoremstyle{plain} \newtheorem*{theorem*}{Theorem}
\theoremstyle{plain} 
\theoremstyle{plain} \newtheorem*{mthm*}{Main Theorem}
\theoremstyle{plain} \newtheorem*{conjecture*}{Conjecture}
\theoremstyle{plain} 
\theoremstyle{plain} \newtheorem*{problem*}{Problem}

\title{Characterization of minimizers of Aviles-Giga functionals in special domains}

\thanks{The author has been supported by the SNF Grant 182565.}

\author[E.~Marconi]{Elio Marconi}
\address{Elio Marconi, EPFL B, Station 8, CH-1015 Lausanne, CH.}
\email{elio.marconi@epfl.ch}

\begin{document}
	\maketitle
	
\begin{abstract}
We consider the singularly perturbed problem $F_\e (u,\Omega):=\int_\Omega \e |\nabla^2u| + \e^{-1}|1-|\nabla u|^2|^2$ on bounded domains $\Omega \subset\R^2$. 
Under appropriate boundary conditions, we prove that if $\Omega$ is an ellipse then the minimizers of $F_\e(\cdot,\Omega)$ converge to the viscosity solution of the eikonal equation $|\nabla u|=1$ as $\e \to 0$.
\end{abstract}

\section{Introduction}
\subsection{The main result}
We consider the family of functionals
\begin{equation}\label{E_functional}
F_\e(u,\Omega):= \int_{\Omega} \left( \e |\nabla^2 u| + \frac{1}{\e}\left|1-|\nabla u|^2\right|^2\right)dx,
\end{equation}
where $\Omega\subset \R^2$ is a $C^2$ bounded open set, $\e>0$ and $u \in W^{2,2}_0(\Omega)$.
These functionals were introduced in \cite{AG_conjecture} and proposed as a model for blistering in \cite{OG_morphology}.
In these cases we are interested in the minimizers $u_\e$ of $F_\e$ in the space
\begin{equation*}
\Lambda(\Omega):= \left\{ u\in W^{2,2}_0(\Omega) : \frac{\partial u}{\partial n}=-1 \mbox{ on }\partial \Omega\right\},
\end{equation*}
where $n$ denotes the outer normal to $\Omega$.
The final goal is the understanding of the behavior of $u_\e$ as $\e \to 0$.
In \cite{OG_morphology} (and more explicitly in \cite{AG_viscosity}) it is conjectured that
\begin{equation}\label{E_conj_lim}
u_\e \to \bar u:= \dist(\cdot, \partial \Omega),
\end{equation}
at least for convex domains $\Omega$. A first partial result in this direction was obtained in \cite{JK_entropies}, where the authors 
proved that if $\Omega$ is an ellipse, then
\begin{equation}\label{E_JK}
\lim_{\e \to 0} \min F_{\e}(\cdot,\Omega) = F_0(\bar u, \Omega),
\end{equation} 
where $F_0$ is the candidate asymptotic functional that we are going to introduce in \eqref{E_def_F0}.

The main result of this paper is the proof of \eqref{E_conj_lim} in the same setting as in \cite{JK_entropies}, namely
\begin{theorem}\label{T_main}
Let $\Omega\subset \R^2$ be an ellipse and, for every $\e>0$, let $u_\e$ be a minimizer of $F_\e(\cdot, \Omega)$. 
Then 
\begin{equation*}
\lim_{\e \to 0}u_{\e} = \dist(\cdot, \partial \Omega) \qquad \mbox{in }W^{1,1}(\Omega).
\end{equation*}
\end{theorem}
This result is obtained as a corollary after showing that $\bar u$ is the unique minimizer of a suitable asymptotic problem for 
$F_\e(\cdot, \Omega)$ as $\e \to 0$. In order to rigorously introduce it, we recall some previous results (see also the introduction of \cite{CDL_limsup} for a presentation of the history of the problem).

\subsection{Previous results}
In the following $\Omega$ denotes a $C^2$ bounded open subset of $\R^2$.
Independently from the validity of \ref{E_conj_lim}, 
it is conjectured already in \cite{AG_conjecture} that 
\begin{enumerate}
\item if $u_\e$ is such that $\limsup_{\e \to 0}F_\e(u_\e,\Omega)< \infty$, then $u_\e$ converges up to subsequences to a Lipschitz solution $u$ of the eikonal equation $|\nabla u|=1$;
\item if $u_\e$ is a sequence of minimizers of $F_\e(\cdot, \Omega)$, then any limit $u$ of $u_\e$ minimizes the functional
\begin{equation}\label{E_def_F0}
F_0(v,\Omega):= \frac{1}{3}\int_{J_{\nabla v}}|\nabla v ^+ - \nabla v^-|^3 d\mathcal H^1,
\end{equation}
among the solutions of the eikonal equation. Here, $J_{\nabla v}$ denotes the jump set of $\nabla v$ and $\nabla v^\pm$ the corresponding
traces.
\end{enumerate}

 A positive answer to the first point was obtained independently in \cite{DMKO_compactness} and \cite{ADLM_eikonal}. 
A fundamental notion in this analysis and in particular in \cite{DMKO_compactness} is the one of entropy, borrowed from the field of conservation laws. 
\begin{definition}
We say that $\Phi \in C^\infty_c(\R^2;\R^2)$ is an \emph{entropy} if for every open set $\Omega\subset \R^2$ and every smooth
$m:\Omega \to \R^2$ it holds
\begin{equation}\label{E_def_entropy}
\left( \div \,m=0  \mbox{ and } |m|^2=1 \right) \quad \Rightarrow \quad \div(\Phi(m))=0.
\end{equation}
We will denote by $\mathcal E$ the set of entropies.
\end{definition}
We will consider the following family of entropies introduced first in \cite{JK_entropies,AG_semicontinuity}:
\begin{equation*}
\Sigma_{\alpha_1,\alpha_2}(z):= \frac{4}{3}\left((z\cdot \alpha_2)^3 \alpha_1 + (z\cdot \alpha_1)^3\alpha_2\right),
\end{equation*}
where $(\alpha_1,\alpha_2)$ is an orthonormal system in $\R^2$. 

 Collecting the results of \cite{DMKO_compactness} and \cite{ADLM_eikonal} we get the following statement.
\begin{theorem}\label{T_compactness}
Let $\e_k \to 0$ and $u_k \in W^{2,2}_0(\Omega)$ be such that $\limsup_{k\to \infty} F_{\e_k}(u_k, \Omega)< \infty$. 
Then $m_k:=\nabla^\perp u_k$ is pre-compact in $L^1(\Omega)$. 
Moreover if $m_k$ converges to $m$ in $L^1(\Omega)$, then $|m|=1$ a.e. in $\Omega$,  for every entropy $\Phi \in \mathcal E$ it holds
\begin{equation*}\label{E_diss}
\mu_\Phi := \div\, \Phi ( m) \in \mathcal M (\Omega),
\end{equation*}
where $\mathcal M(\Omega)$ denotes the set of finite Radon measures on $\Omega$, and
\begin{equation*}
\left(\bigvee_{(\alpha_1,\alpha_2)} |\div\, \Sigma_{\alpha_1,\alpha_2}(m)|\right)(\Omega)\le \liminf_{k\to \infty} F_{\e_k}(u_k,\Omega),
\end{equation*}
where $\bigvee$ denotes the supremum operator on non-negative measures (see for example \cite[Def. 1.68]{AFP_book}).
\end{theorem}

Theorem \ref{T_compactness} motivates
the introduction of the following space of vector fields, which contains all the limits of sequences $\nabla^\perp u_{\e_k}$, where $u_{\e_k}$  have equi-bounded energy.
\begin{definition}
We denote by $A(\Omega)$ the set of all $m\in L^\infty(\Omega; \R^2)$ such that 
\begin{equation*}
\div\, m = 0 \quad \mbox{in }\D'(\Omega), \qquad \qquad |m|^2=1 \quad \L^2\mbox{-a.e. in }\Omega
\end{equation*}
and such that for every entropy $\Phi \in \mathcal E$ it holds
\begin{equation*}
\mu_\Phi := \div\left(\Phi(m)\right)  \in \mathcal M_{\loc} (\Omega),
\end{equation*}
namely $\mu_\Phi$ is a locally finite Radon measure on $\Omega$.
We moreover set
\begin{equation*}
\tilde F_0(u,\Omega):= \left(\bigvee_{(\alpha_1,\alpha_2)} \left|\div {\Sigma_{\alpha_1,\alpha_2}}\left(\nabla^\perp u\right)\right|\right) (\Omega).
\end{equation*}
 Finally we denote by 
\begin{equation*}
\Lambda^0(\Omega):=\left\{ u \in W^{1,\infty}_0(\Omega) : \nabla^\perp u \in A(\Omega)\right\}.
\end{equation*}
\end{definition}
The functional $\tilde F_0(\cdot,\Omega)$ coincides with $F_0(\cdot, \Omega)$ in
the subspace of $\Lambda^0(\Omega)$ whose elements have gradient in $\BV_{\loc}(\Omega)$ (see \cite{ADLM_eikonal}) and it is the 
natural candidate to be the $\Gamma$-limit of the functionals $F_\e(\cdot, \Omega)$ as $\e \to 0^+$.

Although $A(\Omega)\not\subset \BV_\loc(\Omega)$, elements of $A(\Omega)$ share with $\BV$ functions most of their fine properties:
\begin{theorem}[\cite{DLO_JEMS}]\label{T_structure}
For every $m\in A(\Omega)$ there exists a $\H^1$-rectifiable set $J\subset \Omega$ such that
\begin{enumerate}
\item for $\H^1$-a.e. $x \notin J$ it holds
\begin{equation*}
\lim_{r\to 0}\frac{1}{r^2}\int_{B_r(x)}|m(y)-\bar m_{x,r}| dy =0,
\end{equation*}
where $\bar m_{x,r}$ denotes the average of $m$ on $B_r(x)$, namely $x$ is a vanishing mean oscillation point of $m$;
\item for $\H^1$-a.e. $x \in J$ there exist $m^+(x),m^-(x) \in \S^1$ such that
\begin{equation*}
\lim_{r \to 0} \frac{1}{r^2}\left( \int_{B_r^+(x)} |m(y)-m^+(x)|dy + \int_{B_r^-(x)}|m(y)-m^-(y)|dy  \right) =0,
\end{equation*}
where $B^\pm(x):= \{ y \in B_r(x): \pm y \cdot \mathbf{n}(x)>0\}$ and $\mathbf{n}(x)$ is a unit vector normal to $J$ in $x$;
\item for every $\Phi \in \mathcal E$ it holds
\begin{equation*}
\begin{split}
\mu_\Phi \llcorner J & = [ \mathbf{n}\cdot (\Phi(m^+)-\Phi(m^-))] \H^1\llcorner J,\\ 
\mu_\Phi \llcorner K & = 0 \quad \forall K\subset \Omega \setminus J \mbox{ with }\H^1(K)< \infty.
\end{split}
\end{equation*}
\end{enumerate} 
\end{theorem}
The analogy with the structure of elements in $A(\Omega)\cap \BV_\loc(\Omega)$ is not complete: 
for these functions properties (1) and (3) can be improved to
\begin{enumerate}
 \item[(1')] $\H^1$-a.e. $x \notin J$ is a Lebesgue point of $m$;
 \item[(3')]  for every $\Phi \in \mathcal E$
 \begin{equation}\label{E_diss_J}
 \mu_\Phi = [ \mathbf{n}\cdot (\Phi(m^+)-\Phi(m^-))] \H^1\llcorner J.
 \end{equation}
 \end{enumerate}
In order to prove (3') from (3) one should show that $\mu_\Phi$ is concentrated on $J$.
This is considered as a fundamental step towards the solution of the $\Gamma$-limit conjecture and it remains open. 
Notice moreover that by means of Theorem \ref{T_structure} we can give a meaning to the definition of the functional $F_0(\cdot,\Omega)$
even for solutions $u$ to the eikonal equation with $\nabla^\perp u\in A(\Omega)\setminus \BV_\loc(\Omega)$; Property (3') would imply that
$F_0$ coincides with $\tilde F_0$ on the whole $\Lambda^0(\Omega)$.

A fundamental tool in the study of fine properties of elements of $A(\Omega)$ is the kinetic formulation \cite{JP_kinetic} 
(see also \cite{LPT_kinetic} in the framework of scalar conservation laws). Here we use a more recent version obtained in \cite{GL_eikonal}.
\begin{theorem}\label{T_kin}
Let $m\in A(\Omega)$. Then there exists $\sigma \in \M_\loc(\Omega \times \R/2\pi\Z)$ such that
\begin{equation}\label{E_kinetic}
e^{is}\cdot \nabla_x \chi = \partial_s \sigma \qquad \mbox{in }\D'(\Omega \times \R/2\pi\Z),
\end{equation}
where $\chi: \Omega \times \R/2\pi\Z$ is defined by
\begin{equation}\label{E_def_chi}
\chi (x,s) =
\begin{cases}
1 &\mbox{if }e^{is}\cdot m(x)>0, \\
0 &\mbox{otherwise}.
\end{cases}
\end{equation}
\end{theorem}
We observe that if $\sigma$ solves \eqref{E_kinetic}, then also
\begin{equation*}
\sigma + \mu \otimes \mathcal L^1
\end{equation*}
solves \eqref{E_kinetic} for every $\mu\in \M_\loc(\Omega)$. This ambiguity is resolved in \cite{GL_eikonal} by considering the unique $\sigma_0$ solving
\eqref{E_kinetic} such that 
\begin{equation*}\label{E_normalization_GL}
\int_{\Omega \times \S^1} \varphi(x) d \sigma_0(x,s)=0, \qquad \forall \varphi \in C^\infty_c(\Omega).
\end{equation*}
The above kinetic formulation encodes the entropy production of the family of entropies
\begin{equation*}
\E_\pi:=\left\{ \Phi \in \E: \frac{d}{ds}\Phi(e^{is})|_{s=\bar s} = - \frac{d}{ds}\Phi(e^{is})|_{s=\bar s + \pi} \right\}.
\end{equation*} 
Condition \eqref{E_def_entropy} is equivalent to $\frac{d}{ds}\Phi(e^{is}) \cdot e^{is}=0$ for every $s\in \R/2\pi\Z$, 
therefore for every $\Phi \in \E$ we can define $\psi_\Phi:\R/2\pi\Z \to \R$ such that 
\begin{equation*}
 \frac{d}{ds}\Phi(e^{is}) = 2\psi_\Phi\left(s+ \frac{\pi}{2}\right)e^{i\left(s+\frac{\pi}{2}\right)} \qquad \forall s \in \R/2\pi\Z.
\end{equation*}
Notice that $\Phi \in \E_\pi$ if and only if $\psi_\Phi$ is $\pi$-periodic.
Rephrasing the construction in \cite{GL_eikonal}, we have the following identity:
for every $\Phi \in \E_\pi$ and every $\zeta \in C^1_c (\Omega)$ it holds
\begin{equation}\label{E_kin_entropy}
\langle\div \Phi(m), \zeta \rangle = \langle \partial_s \sigma, \zeta \otimes \psi_\Phi \rangle,
\end{equation}
namely
\begin{equation*}
\int_\Omega \Phi(m) \cdot \nabla \zeta dx = \int_{\Omega \times \R/2\pi\Z}\zeta (x) \psi'_\Phi(s) d\sigma.
\end{equation*}

A possibly weaker version of (3') is the following:
\begin{enumerate}
\item[(3'')] Eq. \eqref{E_diss_J} holds for every $\Phi \in \mathcal E_\pi$.
\end{enumerate}
This is equivalent to require that $\nu_0:= (p_x)_\sharp |\sigma_0| \in \mathcal M_\loc(\Omega)$ is concentrated on $J$ and moreover it would be sufficient to establish the equality $F_0=\tilde F_0$. 
The following proposition is a partial result in this direction for general $m \in A(\Omega)$; we remark here that a key step of the proof of Theorem \ref{T_main} is to establish (3'') for a class of $m$ including the limits of $\nabla^\perp u_\e$, where $u_\e$ is a minimizer of $F_\e(\cdot, \Omega)$ and $\Omega$ is an ellipse.

\begin{proposition}\label{P_factorization}
Let $m \in A(\Omega)$ and $(\sigma_{0,x})_{x\in \Omega}\subset \mathcal M(\R/2\pi\Z)$ be the disintegration of $\sigma_0$ with respect to $\nu_0$ defined for $\nu_0$-a.e. $x \in \Omega$ by the properties $|\sigma_{0,x}|(\R/2\pi\Z)=1$ and 
\begin{equation*}
\int_{\Omega \times \R/2\pi\Z}\varphi (x,s) d \sigma_0(x,s) = \int_{\Omega} \int_{\R/2\pi\Z} \varphi(x,s)d \sigma_{0,x}(s) d \nu_0(x)
\end{equation*}
for every $\varphi \in C^\infty_c(\Omega \times \R/2\pi\Z)$.
Then for $\nu_0$-a.e. $x \in \Omega \setminus J$ there exists $\bar s \in \R/2\pi\Z$ such that 
\begin{equation*}
\sigma_{0,x}= \pm \frac{1}{4}\left( \delta_{\bar s} + \delta_{\bar s + \pi} - \frac{1}{\pi} \mathcal L^1\right).
\end{equation*}
\end{proposition}
Among other results, the same expression for $\sigma_{0,x}$ has been obtained very recently in \cite{LP_factorization} under the additional assumption that 
$\div \,\Phi(m) \in L^p(\Omega)$ for every $\Phi \in \mathcal E$. As the authors point out, it is still not known if this additional 
assumption is sufficient to establish that indeed $\sigma_0$ vanishes.

\subsection{The asymptotic problem}

Adapting the argument in \cite{Vasseur_traces} for scalar conservation laws to this context, it is possible to prove that the elements of $A(\Omega)$ with finite energy have 
strong traces in $L^1$ at the boundary of $\Omega$. However, the conditions 
\begin{equation*}
u_\e \in \Lambda (\Omega), \qquad  \limsup_{\e \to 0} F_{\e}(u_\e,\Omega) < \infty, \qquad \mbox{and} \qquad u=\lim_{\e \to 0}u_\e \quad \mbox{in }W^{1,1}
\end{equation*}
do not guarantee that $\frac{\partial u}{\partial n} =-1$ on $\partial \Omega$; in other words we can have boundary layers.
In order to take them into account we slightly reformulate the minimum problem for $F_\e(\cdot,\Omega)$:
given $\delta >0$ we define
\begin{equation*}
\Omega_\delta = \{ x \in \R^2: \dist (x,\Omega) < \delta\}, \qquad \mbox{and} \qquad S_\delta := \Omega_\delta \setminus \bar \Omega. 
\end{equation*}
Being $\Omega$ of class $C^2$, we can take $\delta>0$ sufficiently small so that the function $ - \dist(x, \partial \Omega)$ 
belongs to $W^{2,2}(S_\delta)$.
We therefore consider the minimum problems for the functionals $F_\e(\cdot, \Omega_\delta)$ on the space
\begin{equation*}
\Lambda_\delta(\Omega):= \left\{ u\in W^{2,2}(\Omega_\delta) : u (x)=- \dist(x, \partial \Omega) \mbox{ for a.e. }x \in S_\delta\right\}.
\end{equation*}
Notice that for every $u \in \Lambda(\Omega)$ the function $u^\delta: \Omega_\delta \to \R$ defined by
\begin{equation}\label{E_udelta}
u^\delta(x):=
\begin{cases}
u(x) & \mbox{if }x \in \Omega, \\
- \dist(x, \partial \Omega) & \mbox{if }x \in \Omega_\delta \setminus \Omega
\end{cases}
\end{equation}
belongs to $\Lambda_\delta (\Omega)$ and
\begin{equation*}
F_\e(u^\delta,\Omega_\delta) = F_\e(u,\Omega) + \e \int_{S_\delta} |\nabla ^2  \dist(x, \partial \Omega)|^2 dx.
\end{equation*} 
Similarly the restriction to $\Omega$ of any function in $\Lambda_\delta(\Omega)$ belongs to $\Lambda(\Omega)$, so that the two minimum problems are equivalent.
We will also denote by 
\begin{equation*}
A_\delta(\Omega):= \left\{ m \in A(\Omega_\delta): m = -\nabla^\perp \dist(\cdot, \partial \Omega) \mbox{ in }S_\delta\right\}.
\end{equation*}

We will prove the following result:
\begin{theorem}\label{T_unique}
Let $\Omega$ be an ellipse. Then the function $\bar u^\delta$, defined by \eqref{E_udelta} with $\bar u =  \dist(x, \partial \Omega)$, 
is the unique minimizer of $\tilde F_0(\cdot, \Omega_\delta)$ in the space
\begin{equation*}
\Lambda^0_\delta(\Omega):=\left\{ u \in W^{1,2}(\Omega_\delta) : \nabla^\perp u \in A(\Omega_\delta) \mbox{ and } u= \bar u^\delta\mbox{ in } S_\delta \right\}.
\end{equation*}
\end{theorem}
We show now that Theorem \ref{T_main} is a corollary of Theorem \ref{T_unique} and the previous mentioned results: indeed 
let $\e_k\to 0$ as $k\to \infty$ and for any $k$ let $u_{\e_k}$ be a minimizer of $F_{\e_k}(\cdot, \Omega)$ on 
$\Lambda(\Omega)$.
By Theorem \ref{T_compactness} and \eqref{E_JK} we have that every limit point $u_0$ of $u_{\e_k}$ belongs to $\Lambda^0(\Omega)$ and
moreover it holds
\begin{equation*}
\tilde F_0(u^\delta_0,\Omega_\delta) \le \liminf_{k\to \infty} F_{\e_k}(u^\delta_{\e_k},\Omega_\delta) =  \liminf_{k\to \infty} F_{\e_k}(u_{\e_k},\Omega) = \lim_{k\to \infty} \min_{\Lambda(\Omega)}F_{\e_k}(\cdot,\Omega) = \tilde F_0(\bar u,\Omega)=\tilde F_0(\bar u^\delta,\Omega_\delta).
\end{equation*}
Since $\bar u^\delta$ is the only minimizer of $\tilde F_0(\cdot, \Omega_\delta)$ in $\Lambda^0_\delta(\Omega)$, then $u_0^\delta = \bar u^\delta$, namely $u_0=\bar u$. 

\newpage
\subsection{Related results}
\subsubsection{Zero-energy states} The only case in which the behavior of minimizers of $F_\e(\cdot,\Omega)$ as $\e \to 0$ is completely
understood is when $\lim_{\e\to 0}\min F_\e(\cdot,\Omega)=0$.  
All the sets $\Omega$ admitting sequences with vanishing energy were characterized in \cite{JOP_zero-energy} and with 
the appropriate boundary conditions the limit function is in these cases $\bar u = \dist(\cdot, \partial\Omega)$. 
A quantitative version of this result is proven in \cite{Lorent_quantitative} (see also \cite{Lorent_simple}). 
In a different direction, it was shown in \cite{LP_two_entropies} that the vanishing of the two entropy defect measures 
$\div \Sigma_{e_1,e_2}(m)$ and $\div \Sigma_{\e_1,\e_2}(m)$ is sufficient to establish $\div \,\Phi(m)=0$ for every $\Phi \in \E$. 
Here we denoted by $(e_1,e_2)$ the standard orthonormal system in $\R^2$ and by 
\begin{equation*}
(\e_1,\e_2):= \left( \left( \frac{1}{\sqrt 2}, \frac{1}{\sqrt 2} \right), \left(-\frac{1}{\sqrt 2}, \frac{1}{\sqrt 2}  \right)\right)
\end{equation*}
the orthonormal system obtained by performing a rotation of $(e_1,e_2)$ by $\pi/4$.

\subsubsection{States with a vanishing entropy defect measure}
The case when $\Omega$ is an ellipse is special since we know a priori that there exists an orthonormal system $(\alpha_1,\alpha_2)$ in $\R^2$ for which the minimizers $u^\delta$ in $A(\Omega_\delta)$ of the asymptotic problem $\tilde F_0(\cdot, \Omega_\delta)$ satisfy
\begin{equation}\label{E_vanishing_entropy}
\div \Sigma_{\alpha_1,\alpha_2}\left(\nabla^\perp u^\delta\right)=0 \qquad \mbox{in }\D'(\Omega_\delta).
\end{equation}
This situation has been considered more extensively in \cite{IM_eikonal,Ignat_confluentes}, where in particular the authors proved the minimizing property of the viscosity
solution \eqref{E_JK} for more general domains and functionals. In this direction we only mention here that the same arguments of this paper 
allow to prove Theorem \ref{T_main} also in the case where $\Omega$ is a stadium, namely a domain of the form
\begin{equation*}
\Omega = \{ x \in \R^2: \dist ( x, [0,L]\times \{0\}) < R\} \qquad \mbox{for some }L,R>0.
\end{equation*}
We finally mention that under the additional assumption \eqref{E_vanishing_entropy} we can prove Property (3'').

\subsubsection{A micromagnetics model}
A family of functionals $E_\e$ strictly related to \eqref{E_functional} was introduced in \cite{RS_magnetism,RS_magnetism2} in the context of micro-magnetics. 
An analogous result to Theorem \ref{T_main} was proved in \cite{ALR_viscosity} even for general smooth domains $\Omega$, while the $\Gamma$-limit
conjecture is still open also in this setting. Although Theorem \ref{T_structure} has a perfect analogue for the elements in the asymptotic domain of $E_\e$ (see \cite{AKLR_rectifiability}), the main difficulty seems to be a still not complete understanding of the fine properties of these elements. 
In this direction we notice that the method used here to establish Proposition \ref{P_factorization} gives the analogue in this setting of the concentration property (3') (see \cite{M_RS}).

\section{Lagrangian representation of elements in $A(\Omega)$}
The Lagrangian representation is an extension of the classical method of characteristics to the non-smooth setting: it was introduced in
the framework of scalar conservation laws in \cite{BBM_multid,M_Lebesgue} building on the kinetic formulation from \cite{LPT_kinetic}.
This approach is strongly inspired by the decomposition in elementary solutions of non-negative measure valued solutions of the
linear transport equation, called superposition principle (see \cite{AC_superposition}). Indeed by Theorem \ref{T_kin}, 
the vector fields $m \in A(\Omega)$ are represented by the solution $\chi$ of the linear transport equation \eqref{E_kinetic}. 
The main difficulty in this case is due to the source term which is merely a derivative of a measure.
This issue is reflected in the lack of regularity of the characteristics detected by our Lagrangian representation, which have bounded variation but they are in general not continuous.
A fundamental feature for our analysis is that we can decompose the kinetic measure $\sigma$ in \eqref{E_kinetic} along the characteristics.

\subsection{Lagrangian representation}

We introduce the following space of curves:
given $T>0$ we let
\begin{equation*}
\Gamma:= \left\{ (\gamma,t^-_\gamma,t^+_\gamma): 0\le t^-_\gamma\le t^+_\gamma\le T,
\gamma=(\gamma_x,\gamma_s)\in \BV((t^-_\gamma,t^+_\gamma);\Omega \times \R/2\pi \Z) , \gamma_x \mbox{ is Lipschitz} \right\}.
\end{equation*}
We will always consider the right-continuous representative of the component $\gamma_s$.
Moreover we will adopt the notation from \cite{AFP_book} for the decomposition of the measure $Dv$ where $v \in \BV(I;\R)$ for some interval $I\subset \R$:
\begin{equation*}
Dv = \tilde D v + D^j v,
\end{equation*}
where $\tilde D v$ denotes the sum of the absolutely continuous part and the Cantor part of $Dv$ and $D^jv$ denotes the jump part of $Dv$.
We will need to consider also $\tilde D v$ for functions $v\in \BV (I;\R/2\pi \Z)$. In this case $\tilde Dv = \tilde Dw$ where $w$ is any function
in $\BV(I;\R)$ such that for every $z \in I$ the value $w(z)$ belongs to the class $v(z)$ in $\R/2\pi\Z$.
For every $t \in (0,T)$ we consider the section 
\begin{equation*}
 \Gamma(t):= \left\{\left(\gamma,t^-_\gamma,t^+_\gamma\right)\in  \Gamma: t \in \left(t^-_\gamma,t^+_\gamma\right)\right\}
\end{equation*}
and we denote by 
\begin{equation*}
\begin{split}
 e_t: \Gamma(t) &\to \Omega \times \R/2\pi \Z \\
(\gamma, t^-_\gamma, t^+_\gamma) & \mapsto  \gamma(t).
\end{split}
\end{equation*}

\begin{definition}\label{D_Lagr}
Let $m\in A(\Omega)$ and $\Omega'$ be a $W^{2,\infty}$-open set compactly contained in $\Omega$
We say that a finite non-negative Radon measure $\omega \in \M( \Gamma)$ is a \emph{Lagrangian representation} of $m$ in $\Omega'$ if the following conditions hold:
\begin{enumerate}
\item for every $t\in (0,T)$ it holds
\begin{equation}\label{E_repr_formula}
( e_t)_\sharp \left[ \omega \llcorner  \Gamma(t)\right]= \chi \L^{2}\times \L^1,
\end{equation}
where $\chi$ is defined in \eqref{E_def_chi};
\item the measure $\omega$ is concentrated on curves $(\gamma,t^-_\gamma,t^+_\gamma)\in  \Gamma$ such that for $\L^1$-a.e. $t \in (t^-_\gamma,t^+_\gamma)$ 
the following characteristic equation holds:
\begin{equation}\label{E_characteristic}
\dot\gamma_x(t)= e^{i \gamma_s(t)};
\end{equation}
\item it holds the integral bound
\begin{equation*}\label{E_reg}
\int_{ \Gamma} \TV_{(0,T)} \gamma_s d\omega(\gamma) <\infty;
\end{equation*}
\item for $\omega$-a.e. $(\gamma,t^-_\gamma,t^+_\gamma)\in \Gamma$ it holds
\begin{equation*}
t^-_\gamma>0 \Rightarrow  \gamma_x(t^-_\gamma +) \in \partial \Omega', \qquad \mbox{and} \qquad 
t^+_\gamma<T \Rightarrow  \gamma_x(t^+_\gamma -) \in \partial \Omega'.
\end{equation*}
\end{enumerate}
\end{definition}

For every curve $\gamma \in \Gamma$ we define the measure $\sigma_\gamma \in \M((0,T)\times \Omega' \times \R/2\pi\Z)$ by
\begin{equation}\label{E_def_tildesigmagamma}
 \sigma_\gamma=(\Id, \gamma)_\sharp  \tilde D_t \gamma_s + \H^1\llcorner E_{\gamma}^+ -\H^1\llcorner E_{\gamma}^-,
\end{equation}
where
\begin{equation}\label{E_def_Epmgamma}
\begin{split}
E_\gamma^+ &:= \{(t,x,s) \in (0,T)\times \Omega\times \R/2\pi\Z: \gamma_x(t)=x  \mbox{ and } \gamma_s(t-)\le s \le \gamma_s(t+)\le \gamma_s(t-)+\pi \}, \\
E_\gamma^- & := \{(t,x,s) \in (0,T)\times \Omega\times \R/2\pi\Z: \gamma_x(t)=x  \mbox{ and } \gamma_s(t+)\le s \le \gamma_s(t-)< \gamma_s(t+)+\pi \}.
\end{split}
\end{equation}
Notice that since $\R/2\pi\Z$ is not ordered, given $s_1\ne s_2 \in \R/2\pi\Z$ the condition $s_1< s_2$ is not defined. 
Nevertheless we use the notation $s \in (s_1,s_2)$ or $s_1<s<s_2$ to indicate the following condition (depending only on the orientation of $\R/2\pi\Z$):
if $t_1,t_2 \in \R$ are such $t_1<t_2<t_1+2\pi$, $e^{it_1}=e^{is_1}$ and $e^{it_2}=e^{is_2}$ then there exists $t \in (t_1,t_2)$ such that
$e^{it}=e^{is}$.

\begin{lemma}\label{L_Lagr_meas}
Let $\omega$ be a Lagrangian representation of $m\in A(\Omega)$ on $\Omega'$. Let us denote by
\begin{equation*}
 \sigma_\omega:= -  \int_\Gamma  \sigma_\gamma d\omega
\end{equation*}
and by $\tilde \chi :(0,T)\times \Omega\times \R/2\pi\Z \to \R$ the function defined by
$\tilde \chi(t,x,s)=\chi(x,s)$ for every $t \in (0,T)$.
Then it holds
\begin{equation}\label{E_kin_tilde}
e^{is}\cdot \nabla_x \tilde \chi = \partial_s  \sigma_\omega \quad \in \D'((0,T)\times \Omega'\times \R/2\pi\Z).
\end{equation}
\end{lemma}
\begin{proof}
We show that \eqref{E_kin_tilde} holds when tested with every function of the form $\phi(t,x,s)=\zeta(t)\varphi(x,s)$ with $\zeta \in C^\infty_c((0,T))$ and $\varphi \in C^\infty_c(\Omega'\times \R/2\pi\Z)$.
It follows from \eqref{E_repr_formula} and \eqref{E_characteristic} that
\begin{equation}\label{E_chain1}
\begin{split}
\langle e^{is}\cdot \nabla_x \tilde \chi, \phi\rangle = &~ - \int e^{is}\cdot \nabla_x \varphi(x,s) \zeta(t) \tilde \chi(t,x,s)dtdxds \\
=&~- \int_{(0,T)}\int_{\Gamma(t)} e^{i\gamma_s(t)}\cdot \nabla_x \varphi( \gamma(t)) d\omega  \zeta(t) dt \\
=&~- \int_\Gamma \int_{t^-_\gamma}^{t^+_\gamma} \frac{d}{dt} \gamma_x(t) \cdot \nabla_x \varphi (\gamma(t)) \zeta(t) dt d\omega.
\end{split}
\end{equation}
By the chain-rule for functions with bounded variation we have the following equality between measures:
\begin{equation*}
\frac{d}{dt}\varphi \circ \gamma =  \nabla_x \varphi (\gamma(t)) \cdot \frac{d}{dt} \gamma_x(t)  
+ \partial_s\varphi (\gamma(t)) \tilde D_t \gamma_s + \sum_{t_j \in J_\gamma} \left( \varphi (t_j,\gamma(t_j+)) -  \varphi (t_j,\gamma(t_j-))\right) \delta_{t_j},
\end{equation*}
where $J_{\gamma}$ denotes the jump set of $\gamma$.
Therefore, proceeding in the chain \eqref{E_chain1}, we have
\begin{equation*}\label{E_chain2}
\begin{split}
\langle e^{is}\cdot \nabla_x \tilde \chi, \phi\rangle = &~ 
- \int_\Gamma \int_{t^-_\gamma}^{t^+_\gamma} \frac{d}{dt} \varphi(\gamma(t)) \zeta(t) dt d\omega
+ \int_\Gamma \int_{t^-_\gamma}^{t^+_\gamma} \partial_s\varphi (\gamma(t))\zeta(t) d \tilde D_t \gamma_s(t) d\omega \\
& ~ + \int_\Gamma  \sum_{t_j \in J_\gamma}  \big( \varphi (t_j,\gamma(t_j+)) -  \varphi (t_j,\gamma(t_j-))\big) \zeta(t_j) d\omega.
\end{split}
\end{equation*}
By definition of $\sigma_\gamma$ in \eqref{E_def_tildesigmagamma} it holds
\begin{equation*}
 \int_{t^-_\gamma}^{t^+_\gamma} \partial_s\varphi (\gamma(t))\zeta(t) d \tilde D_t \gamma_s(t) 
 +  \sum_{t_j \in J_\gamma}  \big( \varphi (t_j,\gamma(t_j+)) -  \varphi (t_j,\gamma(t_j-))\big) \zeta(t_j) d\omega = \int \partial_s \varphi (x,s) \zeta(t) d  \sigma_\gamma,
\end{equation*}
therefore in order to establish  $\langle e^{is}\cdot \nabla_x \tilde \chi, \phi\rangle = \langle \partial_s  \sigma_\omega, \phi \rangle$ it suffices to prove that 
\begin{equation*}
\int_\Gamma \int_{t^-_\gamma}^{t^+_\gamma} \frac{d}{dt} \varphi(\gamma(t)) \zeta(t) dt d\omega = 0.
\end{equation*}
By Point (4) in Definition \ref{D_Lagr} for $\omega$-a.e. $\gamma \in \Gamma$ it holds $\varphi(\gamma (t^-_\gamma  +))=  \varphi(\gamma (t^+_\gamma  -))=0$ and in particular 
\begin{equation*}
\begin{split}
\int_\Gamma \int_{t^-_\gamma}^{t^+_\gamma} \frac{d}{dt} \varphi(\gamma(t)) \zeta(t) dt d\omega = &~ -\int_\Gamma \int_{t^-_\gamma}^{t^+_\gamma}\varphi(\gamma(t)) \zeta'(t) dt d\omega \\
=&~ \int_{(0,T)\times \Omega \times \R/2\pi\Z} \tilde \chi \varphi (x,s)\zeta'(t) dt dx ds \\
=&~ 0,
\end{split}
\end{equation*}
where we used \eqref{E_repr_formula} in the second equality and that $\tilde \chi$ does not depend on $t$ in the last equality.
This concludes the proof.
\end{proof}

\begin{definition}
We say that $\sigma \in \M_{\loc}(\Omega'\times \R/2\pi\Z)$ is a \emph{minimal kinetic measure} if it satisfies \eqref{E_kinetic} and for every $\sigma'$
solving \eqref{E_kinetic} it holds 
\begin{equation*}
\nu_\sigma := (p_x)_\sharp |\sigma| \le (p_x)_\sharp|\sigma'| =: \nu_{\sigma'}. 
\end{equation*}
We moreover say that $\omega$ is a a \emph{minimal Lagrangian representation} 
of $m$ if it is a Lagrangian representation of $m$ according to Def. \ref{D_Lagr} and 
\begin{equation*}
\tilde \sigma_\omega = \L^1 \otimes \sigma_t
\end{equation*}
with $\sigma_t$ minimal kinetic measure for $\L^1$-a.e. $t\in (0,T)$.
\end{definition}

The existence of a minimal kinetic measure is proven the following lemma.
\begin{lemma}\label{L_minimal}
For every $m\in A(\Omega)$ there exists a minimal kinetic measure $\sigma$. Moreover there exists $\nu_{\min}\in \M_{\loc}(\Omega)$ such that for every 
minimal kinetic measure $\sigma$ it holds $\nu_{\min} = (p_x)_\sharp |\sigma|$.
\end{lemma}
\begin{proof}
Since $\partial_s \sigma$ is uniquely determined by \eqref{E_kinetic}, we have that a kinetic measure $\sigma$ is minimal if and only if for 
$\nu_\sigma$-a.e. $x \in \Omega$ the disintegration $\sigma_x$ satisfies the following inequality:
\begin{equation}\label{E_min}
1= \|\sigma_x \| \le \left\|\sigma_x + \alpha \L^1\right\| \qquad \forall \alpha \in \R.
\end{equation}
Therefore all minimal kinetic measures are of the form 
\begin{equation*}
\nu_{\sigma_0} \otimes \left( (\sigma_0)_x + \alpha(x) \mathcal L^1 \right),
\end{equation*}
where $\alpha : \Omega \to \R$ is a measurable function such that for $\nu_{\sigma_0}$-a.e. $x \in \Omega$ it holds
\begin{equation}\label{E_minimal_kinetic}
\left\| (\sigma_0)_x + \alpha(x) \mathcal L^1 \right\| \le \left\|(\sigma_0)_x + c\mathcal L^1\right\| \qquad \forall c \in \R.
\end{equation}
The existence of such an $\alpha$ is trivial and in particular it holds
\begin{equation*}
\nu_{\min} = \left(\min_{\alpha \in \R} \left\| (\sigma_0)_x + \alpha \mathcal L^1 \right\|\right) \nu_0. \qedhere
\end{equation*} 
\end{proof}
In Section \ref{S_structure} we will show that for every $m\in A(\Omega)$ there exists a \emph{unique} minimal kinetic measure $\sigma_{\min}$, namely that for $\nu_{\min}$-a.e. $x \in \Omega$ there exists a unique $\alpha(x)$ such that \eqref{E_minimal_kinetic} holds.

The main result of this section is the following:
\begin{proposition}\label{P_Lagr}
Let $\Omega\subset \R^2$ be a bounded open set, $m\in A(\Omega)$ 
and $\Omega'$ be a $W^{2,\infty}$ open set compactly contained in $\Omega$ be such that $\H^1$-a.e. $x \in \partial \Omega'$
is a Lebesgue point of $m$.
Then there exists a minimal Lagrangian representation $\omega$ of $m$ on $\Omega'$. In particular it holds
\begin{equation}\label{E_moduli}
|\sigma_\omega|=\int_\Gamma |\sigma_\gamma|d\omega.
\end{equation}
\end{proposition}
The existence of a Lagrangian representation for weak solutions with finite entropy production to general conservation laws on the whole $(0,T)\times \R^d$ has been proved in \cite{M_Lebesgue}. The case of bounded domains when $\Omega'$ is a ball was considered in \cite{M_RS} for the class of solutions to the eikonal equation arising in \cite{RS_magnetism2}. The extension to the case where $\Omega'$ is a $W^{2,\infty}$ open set does not cause any significant difficulty.
In particular  the argument proposed in \cite{M_RS} applies here with trivial modifications and leads to the following partial result:
\begin{lemma}\label{L_old}
In the setting of Proposition \ref{P_Lagr}, let $\sigma\in \M_{\loc}(\Omega\times \R/2\pi\Z)$ be a locally finite measure satisfying  \eqref{E_kinetic}.
 Then there exists a Lagrangian representation $\omega$ of $m$ on $\Omega'$ such that 
 \begin{equation*}
 \int_\Gamma \TV_{(t^-_\gamma,t^+_\gamma)}\gamma_s d\omega \le T |\sigma|(\Omega' \times \R/2\pi\Z).
 \end{equation*}
\end{lemma}

We now prove Proposition \ref{P_Lagr} relying on Lemma \ref{L_old}.
\begin{proof}[Proof of Proposition \ref{P_Lagr}]
Let $m\in A(\Omega)$ and let $\bar \sigma$ be a minimal kinetic measure. By Lemma \ref{L_old}, there exists a Lagrangian representation 
$\omega$ of $m$ such that
\begin{equation}\label{E_2.1}
 \int_\Gamma \TV_{(t^-_\gamma,t^+_\gamma)}\gamma_s d\omega \le T \|\bar \sigma\|.
 \end{equation}
 By definition of $\sigma_\omega$ it holds
 \begin{equation}\label{E_2.2}
 \|\sigma_\omega\|\le \left(\int_\Gamma |\sigma_\gamma|d\omega\right) \big((0,T)\times \Omega \times \R/2\pi\Z\big) = \int_\Gamma \TV_{(t^-_\gamma,t^+_\gamma)}\gamma_s d\omega.
 \end{equation}
 By Lemma \ref{L_Lagr_meas}, the measure $ \sigma_\omega$ satisfies \eqref{E_kin_tilde}; being $\bar \sigma$ a minimal kinetic measure for $m$, it follows that $T\|\bar \sigma\|\le \| \sigma_\omega\|$. In particular the inequalities in \eqref{E_2.1} and \eqref{E_2.2} are equalities and 
 \eqref{E_moduli} follows.
\end{proof}

The following lemma is a simple application of Tonelli theorem and \eqref{E_repr_formula}; since it is already proven in \cite{M_Burgers}, we refer to it for the details.
\begin{lemma}
For $\omega$-a.e. $(\gamma,t^-_\gamma,t^+_\gamma) \in  \Gamma$ it holds that for $\L^1$-a.e. $t\in (t^-_\gamma,t^+_\gamma)$
\begin{enumerate}
\item $\gamma_x(t)$ is a Lebesgue point of $m$;
\item $e^{i\gamma_s(t)} \cdot m(\gamma_x(t))>0$.
\end{enumerate}
We denote by $ \Gamma_g$ the set of curves $\gamma\in  \Gamma$ such that the two properties above hold.
\end{lemma}

\section{Structure of the kinetic measure}\label{S_structure}

The main goal of this section is to prove Proposition \ref{P_factorization}. As a corollary we will obtain the concentration property (3'') presented in the introduction for solutions $m\in A(\Omega)$ with a vanishing entropy defect measure.
The key step is the following regularity result. 
The strategy of the proof is borrowed from \cite{M_RS}, where an analogous statement was proved for the solutions to the eikonal equation arising in the micromagnetics model mentioned in the introduction. We finally observe that in that situation this result is sufficient to establish the concentration property (3'), while it is not the case here. 
\begin{lemma}\label{L_duality}
Let $\bar \gamma \in \Gamma_g$ and $\bar t \in (t^-_{\bar \gamma},t^+_{\bar \gamma})$, and set $\bar x:=\bar \gamma_x(\bar t)$ and $\bar s:= \bar \gamma_s(t+)$. Then there exists $c>0$ such that for every 
$\delta \in (0,1/2)$ we have at least one of the following:
\begin{enumerate}
\item the lower density estimate holds true:
\begin{equation*}
\liminf_{r\to 0} \frac{\L^2\left(\left\{x \in B_r(\bar x): e^{i \bar s}\cdot m(x) > -\delta \right\}\right)}{r^2} \ge c\delta;
\end{equation*}
\item the following lower bound holds true:
\begin{equation*}
\limsup_{r\to 0}\frac{\nu_{\min}(B_r(\bar x))}{r}\ge c\delta^3.
\end{equation*}
\end{enumerate}
The same statement holds by setting $\bar s := \bar \gamma_s(t-)$.
\end{lemma}

\begin{proof}
We prove the lemma only for $\bar s= \bar \gamma_s(t+)$, being the case  $\bar s = \bar \gamma_s(t-)$ analogous. 
Let $\delta_1>0$ be sufficiently small so that for $\L^1$-a.e. $t\in (\bar t,\bar t +  \delta_1)$ it holds
\begin{equation}\label{E_bars}
e^{i \bar \gamma_s(t)}\cdot e^{i\bar s}\ge \cos \left( \frac{\delta}{5}\right).
\end{equation}
Since $\bar \gamma_x$ satisfies \eqref{E_characteristic}, then for every $r \in \left(0,\frac{\delta_1}{2}\right)$ there exists $t_r \in (\bar t, \bar t + \delta_1)$
such that 
\begin{equation*}
\bar \gamma_x(t) \in B_r(\bar x) \quad \forall t \in (\bar t, t_r), \qquad \mbox{and} \qquad \bar \gamma_x(t_r)\in \partial B_r(\bar x).
\end{equation*}
Moreover since $\cos (\delta/5) \in (1/2,1)$, then \eqref{E_bars} implies
\begin{equation*}
r\le t_r -\bar t \le 2r.
\end{equation*}
For every $r \in \left(0,\frac{\delta_1}{2}\right)$ we denote by
\begin{equation*}
\begin{split}
E_+(r)&:= \{ t \in (\bar t,t_r) : m(\bar \gamma_x(t))\cdot e^{i(\bar \gamma_s(t)+\delta)}>0\}, \\
E_-(r)&:= \{ t \in (\bar t,t_r) : m(\bar \gamma_x(t))\cdot e^{i(\bar \gamma_s(t)-\delta)}>0\}.
\end{split}
\end{equation*}
Since $\gamma \in \Gamma_g$, for $\L^1$-a.e. $t\in (0,t_r)$ it holds
\begin{equation*}
m(\bar \gamma_x(t))\cdot e^{i\bar \gamma_s(t)}>0,
\end{equation*}
therefore, being $\delta \in \left(0,\frac{1}{2}\right)$, we have
\begin{equation*}
(\bar t, t_r)\subset E_+(r)\cup E_-(r).
\end{equation*}
In particular 
\begin{equation*}
\L^1(E_+(r)) + \L^1(E_-(r)) \ge t_r -\bar t\ge r.
\end{equation*}
In the remaining part of the proof we assume that $\L^1(E_-(r))>r/2$, being the case $\L^1(E_+(r))>r/2$ analogous.

Given $\e>0$, we consider the strip
\begin{equation}\label{E_def_Sre}
S_{r,\e}:=\left\{ x \in \Omega_\delta: \exists t \in (\bar t, t_r): \left| \bar\gamma_x(t)-x\right|<\e \right\}.
\end{equation}

For every $(\gamma,t^-_\gamma,t^+_\gamma) \in  \Gamma$ let
$(t^-_{\gamma,i},t^+_{\gamma,i})_{i=1}^{N_{\gamma}}$ be the nontrivial interiors of the  connected components of $\gamma_s^{-1}\left(\left(\bar s - \delta\right), \bar s - \frac{2}{5}\delta\right)$ which intersect $\gamma^{-1}\left(S_{r,\e}\times \left(\bar s - \frac{4}{5}\delta, \bar s - \frac{3}{5}\delta\right)\right)$.
Notice that we have the estimate
\begin{equation*}
N_\gamma \le 1 + \frac{5}{\delta}\TV \gamma_s.
\end{equation*}

For every $i\in \N$ we consider
\begin{equation*}
 \Gamma_i:= \{(\gamma, t^-_\gamma,t^+_\gamma)\in  \Gamma : N_{\gamma}\ge i\}
\end{equation*}
and the measurable restriction map
\begin{equation*}
\begin{split}
R_{i}:  \Gamma_{i} & \to  \Gamma. \\
(\gamma,t^-_\gamma,t^+_\gamma) & \mapsto (\gamma, t^-_{\gamma,i},t^+_{\gamma,i})
\end{split}
\end{equation*}
We finally consider the measure 
\begin{equation*}
\tilde \omega:= \sum_{i=1}^\infty (R_{i})_\sharp \left(\omega\llcorner  \Gamma_i\right).
\end{equation*}
We observe that $\tilde \omega\in \M_+( \Gamma)$ since for every $N>0$
\begin{equation*}
\left\| \sum_{i=1}^N (R_{i})_\sharp \left(\omega\llcorner  \Gamma_i\right) \right\| \le \int_{ \Gamma} N_{\gamma} d\omega \le  \int_{ \Gamma} \left(1 + \frac{5}{\delta}\TV \gamma_s\right) d\omega(\gamma) < \infty
\end{equation*}
by Point (3) in Definition \ref{D_Lagr}.
The advantage of the measure $\tilde \omega$ is that it is concentrated on curves whose $x$-components are transversal to $\bar \gamma_x$
on the whole domain of definition. This property allows to prove the following claim.

{\bf Claim 1}. There exists an absolute constant $\tilde c>0$ such that for $\tilde \omega$-a.e. $(\gamma,t^-_\gamma,t^+_\gamma)\in \Gamma$
it holds
\begin{equation*}
\L^1\left( \left\{ t \in (t^-_\gamma,t^+_\gamma) : \gamma(t) \in S_{r,\e}\times \left( \bar s - \frac{4}{5}\delta, \bar s -\frac{3}{5}\delta \right)\right\} \right)\le \tilde c \frac{\e}{\delta}.
\end{equation*} 
Proof of Claim 1. It follows from \eqref{E_bars} and the characteristic equation \eqref{E_characteristic} that there exists a Lipschitz function 
$f_{\bar \gamma}:\R\to \R$ such that
\begin{equation}\label{E_def_fbargamma}
\big\{\bar \gamma_x(t):t\in (\bar t, \bar t + \delta_1)\big\} \subset 
\left\{ z e^{i\bar s} + f_{\bar \gamma}(z) e^{i\left(\bar s + \frac{\pi}{2}\right)}: z \in \R \right\} \qquad \mbox{and} \qquad \Lip(f_{\bar \gamma})\le \tan\left( \frac{\delta}{5}\right).
\end{equation} 
Similarly for $\tilde \omega$-a.e. $(\gamma,t^-_\gamma,t^+_\gamma) \in \Gamma$ there exists a Lipschitz function $f_\gamma$ such that
\begin{equation*}
\big\{\gamma_x(t):t\in (t^-_\gamma,  t^+_\gamma)\big\} \subset 
\left\{ z e^{i\bar s} + f_{\gamma}(z) e^{i\left(\bar s + \frac{\pi}{2}\right)}: z \in \R \right\} \qquad \mbox{and} \qquad \frac{d}{dz}f_\gamma(z) \in \left(-\tan \delta, -\tan\left( \frac{2}{5}\delta\right)\right)
\end{equation*}
for $\L^1$-a.e. $z \in \R$. By the definitions of $S_{r,\e}$ in \eqref{E_def_Sre} and of $f_{\bar \gamma}$ in \eqref{E_def_fbargamma}, it easily follows that
\begin{equation}\label{E_Sre2}
S_{r,\e}\subset \left\{ x \in \Omega_\delta: f_{\bar \gamma} \left(x\cdot e^{i\bar s}\right) - \e \left( \cos \left(\frac{\delta}{5}\right)\right)^{-1} \le  x\cdot e^{i\left(\bar s + \frac{\pi}{2}\right)}  \le f_{\bar \gamma} \left(x\cdot e^{i\bar s}\right) + \e \left( \cos \left(\frac{\delta}{5}\right)\right)^{-1}  \right\}.
\end{equation}
Given $(\gamma,t^-_\gamma,t^+_\gamma)\in \Gamma$ let us consider the function $g_\gamma:(t^-_\gamma,t^+_\gamma)\to \R$ defined by
\begin{equation*}
g_\gamma(t)= \gamma_x(t)\cdot e^{i\left(\bar s + \frac{\pi}{2}\right)}.
\end{equation*}
By construction of $\tilde \omega$, for $\tilde \omega$-a.e. $(\gamma,t^-_\gamma,t^+_\gamma)\in \Gamma$ and $\L^1$-a.e. $t \in (t^-_\gamma,t^+_\gamma)$ it holds
\begin{equation}\label{E_diffg}
\frac{d}{dt}g_\gamma(t) \le - \sin \left( \frac{2}{5}\delta\right).
\end{equation}
On the other hand 
\begin{equation}\label{E_difff}
\frac{d}{dt}f_{\bar \gamma}(\gamma_x(t)\cdot e^{i\bar s}) \ge - \sin \left( \frac{\delta}{5}\right).
\end{equation}
By \eqref{E_Sre2}, for every $t\in (t^-_\gamma,t^+_\gamma)$ such that  $\gamma_x(t) \in S_{r,\e}$ it holds
\begin{equation*}
f_{\bar \gamma} (\gamma_x(t)\cdot e^{i\bar s}) - \e \left( \cos \left(\frac{\delta}{5}\right)\right)^{-1} \le g_\gamma(t)  \le f_{\bar \gamma} (\gamma_x(t)\cdot e^{i\bar s}) + \e \left( \cos \left(\frac{\delta}{5}\right)\right)^{-1}.
\end{equation*}
Therefore, by \eqref{E_diffg} and \eqref{E_difff}, we have
\begin{equation*}
\L^1\left( \{t: \gamma_x(t) \in S_{r,\e}\} \right) \le \frac{2\e  \left( \cos \left(\frac{\delta}{5}\right)\right)^{-1} }{\left| \sin \left( \frac{2}{5}\delta\right) - \sin \left( \frac{\delta}{5}\right)\right|} \le \tilde c\frac{\e}{\delta},
\end{equation*}
for some universal $\tilde c>0$. This concludes the proof of the claim.

By construction we have
\begin{equation*}
(e_t)_\sharp \tilde \omega \ge \L^3 \llcorner \left\{(x,s) \in S_{r,\e}\times \left( \bar s - \frac{4}{5}\delta, \bar s - \frac{3}{5}\delta\right) : m(x)\cdot e^{i s}>0\right\}
\end{equation*}
for every $t \in (0,T)$. Therefore
\begin{equation}\label{E_strip1}
\begin{split}
 T \L^3  &~\left(\left\{(x,s) \in S_{r,\e}\times  \left( \bar s - \frac{4}{5}\delta, \bar s - \frac{3}{5}\delta\right) : m(x)\cdot e^{i s}>0\right\}\right) \\
&~ \le \int_{\Gamma}\L^1 \left( \left\{t: \gamma(t) \in S_{r,\e}\times \left( \bar s - \frac{4}{5}\delta, \bar s - \frac{3}{5}\delta\right) \right\} \right) d\tilde \omega \\
&~ \le \tilde c \frac{\e}{\delta}\tilde \omega (\Gamma).
\end{split}
\end{equation}
On the other hand, since $\bar \gamma\in \Gamma_g$ and $\L^1(E_-(r))>r/2$ there exists $\bar \e>0$ such that for every $\e \in (0,\bar \e)$ it holds
\begin{equation}\label{E_strip2}
\begin{split}
 \L^3  \left(\left\{(x,s) \in S_{r,\e}\times \left( \bar s - \frac{4}{5}\delta, \bar s - \frac{3}{5}\delta\right) : m(x)\cdot e^{i s}>0\right\}\right) \ge &~
 \frac{1}{2} \L^3  \left( S_{r,\e}\times \left( \bar s - \frac{4}{5}\delta, \bar s - \frac{3}{5}\delta\right) \right) \\ \ge &~
 \frac{\e r \delta}{5}.
 \end{split}
\end{equation}
By \eqref{E_strip1} and \eqref{E_strip2} it follows
\begin{equation*}
\tilde \omega (\Gamma) \ge \frac{\e r \delta}{5} \cdot  \frac{\delta T }{\tilde c \e} = \frac{r\delta^2}{5 \tilde c}T.
\end{equation*}
We consider the split $\Gamma = \Gamma_> \cup \Gamma_<$, where  
\begin{equation*}
\Gamma_>:=\{(\gamma,t^-_\gamma,t^+_\gamma) \in \Gamma : t^+_\gamma- t^-_\gamma \ge r \}, \qquad \mbox{and} \qquad 
\Gamma_<:=\{(\gamma,t^-_\gamma,t^+_\gamma) \in \Gamma : t^+_\gamma- t^-_\gamma < r \}.
\end{equation*}
We will prove the following claim, from which the lemma follows immediately.

\noindent {\bf Claim 2}. There exists an absolute constant $c_1>0$ such that the two following implications hold true:
\begin{enumerate}
\item if $\tilde \omega (\Gamma_>)\ge \frac{r\delta^2T}{10 \tilde c}$, then 
\begin{equation*}
\L^2\left( \left\{ x \in B_{2r} (\bar x) : e^{i\bar \gamma_s(t+)}\cdot m(x)>-\delta \right\} \right)\ge c_1\delta r^2;
\end{equation*}
\item if $\tilde \omega (\Gamma_<)\ge \frac{r\delta^2T}{10 \tilde c}$, then 
\begin{equation*}
\nu(B_{2r}(\bar x)) \ge c_1 \delta^3 r.
\end{equation*}
\end{enumerate}
\emph{Proof of (1)}.
By definition of $\Gamma_>$ and the assumption in (1) we have
\begin{equation*}
\begin{split}
T \frac{r^2\delta^2}{10 \tilde c} \le &~ \int_{\Gamma}\L^1\left( \left\{ t \in (t^-_\gamma,t^+_\gamma): \gamma(t) \in B_{2r}(\bar x)\times \left( \bar s - \delta, \bar s - \frac{2}{5}\delta \right) \right\} \right) d\tilde \omega \\
\le &~ T \L^3\left(\left\{(x,s) \in B_{2r}(\bar x) \times \left( \bar s - \delta, \bar s - \frac{2}{5}\delta \right) : m(x)\cdot e^{is}> 0 \right\}\right) \\
\le &~ T\delta \L^2\left(\left\{x \in B_{2r}(\bar x) : m(x)\cdot e^{i \bar s}>-\delta\right\}\right).
\end{split}
\end{equation*}
\emph{Proof of (2)}.
For $\tilde \omega$-a.e. $(\gamma,t^-_\gamma,t^+_\gamma)\in \Gamma_<$, the image of $\gamma_x$ is contained in $B_{2r}(\bar x)$ and 
$\TV(\gamma_s)\ge \frac{\delta}{5}$. Since $\omega$ is a minimal Lagrangian representation, this implies that
\begin{equation*}
T\nu_{\min}(B_{2r}(\bar x)) = |\sigma_\omega|((0,T)\times B_{2r}(\bar x)) \ge \int_{\Gamma_<}\TV\gamma_s d\tilde \omega \ge \frac{\delta}{5}\tilde \omega(\Gamma_<) \ge \frac{Tr\delta^3}{50\tilde c}. \qedhere
\end{equation*}
\end{proof}

\begin{proposition}\label{P_partial}
Let $m\in A(\Omega)$ and $\sigma\in \M(\Omega \times \R/2\pi\Z)$ be a minimal kinetic measure.
Then for $\nu_{\min}$-a.e. $x \in \Omega\setminus J$ it holds
\begin{equation}\label{E_supp}
\supp \partial_s\sigma_x = \{s,s+\pi\} \qquad \mbox{for some } s \in \R/2\pi\Z. 
\end{equation}
\end{proposition}
\begin{proof}
Let $\omega$ be a minimal Lagrangian representation and let $s,s' \in \R/2\pi\Z$; from the explicit expression of $\sigma_\omega$ 
we have that for $\mathcal L^1 \times \nu_{\min}$-a.e. $(t,x)\in (0,T)\times \Omega$ such that $\supp ( \partial_s(\tilde \sigma_\omega)_{t,x}))\cap (s,s')\ne 0$ 
there exists $(\gamma,t^-_\gamma,t^+_\gamma)\in \Gamma_g$ such that
\begin{equation*}
t \in (t^-_\gamma,t^+_\gamma), \qquad \gamma_x(t)=x, \qquad \mbox{and} \qquad \big[\,  \gamma_s(t-)\in (s,s') \quad \mbox{or} \quad   \gamma_s(t+)\in (s,s') \big].
\end{equation*}
Given $s_1,s_2 \in \pi \Q /2\pi\Z$ with $s_1\ne s_2$ and $s_1 \ne s_2 + \pi$, we set
\begin{equation*}
\delta_{s_1,s_2}= \frac{1}{3}\min\left\{ |s_1-s_2|, |s_1 + \pi -s_2| \right\}
\end{equation*}
so that the intervals $I_1:=(s_1-\delta_{s_1,s_2}, s_1 + \delta_{s_1,s_2})$ , $I_2:=(s_2-\delta_{s_1,s_2}, s_2 + \delta_{s_1,s_2})$,
$I_3:=(s_1+\pi -\delta_{s_1,s_2}, s_1 +\pi + \delta_{s_1,s_2})$ and $I_4:=(s_2+\pi -\delta_{s_1,s_2}, s_2 +\pi + \delta_{s_1,s_2})$ are pairwise disjoint and the distance between any two of these intervals is $\delta_{s_1,s_2}$.
We denote by
\begin{equation*}
E(s_1,s_2):= \left\{ (t,x) \in (0,T)\times \Omega:  \supp ( \partial_s(\sigma_\omega)_{t,x})) \cap I_j \ne \emptyset \mbox{ for } j=1,2,3,4 \right\}.
\end{equation*}

It was shown in \cite{GL_eikonal} that the constraint forces $\sigma$ to be $\pi$-periodic in $s$, in particular for $\mathcal L^1 \times \nu_{\min}$-a.e. $(t,x)\in (0,T)\times \Omega$  the support of $\partial_s \sigma_\omega$ is $\pi$-periodic. Therefore if $(t,x)\in (0,T)\times \Omega$ is such that \eqref{E_supp} does not hold, then there exist four distinct points $\bar s_1,\bar s_2,\bar s_1 + \pi, \bar s_2 + \pi \in \R/2\pi\Z$ belonging to $\supp  ( \partial_s( \sigma_\omega)_{t,x})$.
In particular  $\mathcal L^1 \times \nu_{\min}$-a.e. $(t,x)\in (0,T)\times \Omega$ for which \eqref{E_supp} does not hold belongs to
\begin{equation*}
\bigcup_{s_1,s_2 \in \pi \Q/2\pi\Z} E(s_1,s_2).
\end{equation*}
By the discussion at the beginning of the proof, we have that for $\mathcal L^1 \times \nu_{\min}$-a.e. $(t,x)\in E(s_1,s_2)$ 
and every $j=1,2,3,4$ there exists  $(\gamma_j,t^-_{\gamma_j},t^+_{\gamma_j})\in \Gamma_g$ such that
\begin{equation*}
t \in (t^-_{\gamma_j},t^+_{\gamma_j}), \qquad (\gamma_j)_x(t)=x, \qquad \mbox{and} \qquad \big[\,  (\gamma_j)_s(t-)\in I_j \quad \mbox{or} \quad   (\gamma_j)_s(t+)\in I_j) \big].
\end{equation*}
We show that if $(t,x) \in E(s_1,s_2)$, then $x$ is not a vanishing mean oscillation point of $m$.
Let us assume by contradiction that $x$ is a VMO point of $m$ and there exists $t \in (0,T)$ such that $(t,x)\in E(s_1,s_2)$; by applying Lemma \ref{L_duality} for every $j=1,2,3,4$ there exists $\bar s_j \in I_j$ such that
\begin{equation*}
\liminf_{r\to 0} \frac{\L^2(\{x' \in B_r( x): e^{i \bar s_j}\cdot m(x') > -\delta_{s_1,s_2} \})}{r^2} \ge c\delta_{s_1,s_2}.
\end{equation*} 
Since it does not exist any value $\bar m \in \R^2$ with $|\bar m|=1$ such that $\bar m \cdot e^{i\bar s_j}> - \delta_{s_1,s_2}$ for every $j=1,2,3,4$, this proves that $x$ is not a vanishing mean oscillation point of $m$.
Thm \ref{T_structure} implies that $\mathcal H^1$-a.e. $x \in \Omega \setminus J$ is a VMO point of $m$, therefore since 
$\nu_{\min}\ll \mathcal H^1$, then the set of points $x \in \Omega \setminus J$ for which there exists $t\in (0,T)$ such that 
$(t,x)\in E(s_1,s_2)$ is $\nu_{\min}$-negligible.  

Letting $s_1,s_2$ vary in $\pi \Q/2\pi\Z$, this proves the claim.
\end{proof}

\begin{remark}
Proposition \ref{P_factorization} states for the measure $\sigma_0$ the same property we obtained here for a minimal kinetic measure 
$\sigma$. Although $\sigma_0$ is not always a minimal kinetic measure, the two statements are equivalent since 
$\nu_{\min}\le \nu_0 \ll \nu_{\min}$ and $\partial_s \sigma_0 = \partial_s \sigma$ (see the discussion in Lemma \ref{L_minimal}).
\end{remark}

\begin{corollary}\label{C_disintegration}
For every $m\in A(\Omega)$ there exists a unique minimal kinetic measure $\sigma_{\min}$ of $m$.
In particular for every minimal Lagrangian representation $\omega$ of $m$ on $\Omega' \subset \Omega$ it holds
\begin{equation*}
\sigma_\omega= \L^1\llcorner (0,T)\otimes \sigma_{\min}\llcorner \Omega'.
\end{equation*}
Moreover the disintegration of $\sigma_{\min}$ with respect to $\nu_{\min}$ has
the following structure: 
\begin{enumerate}
\item for $\nu_{\min}$-a.e. $x \in \Omega\setminus J$ it holds 
\begin{equation*}
(\sigma_{\min} )_x =\frac{1}{2} (\delta_{\bar s-\frac{\pi}{2}} + \delta_{\bar s + \frac{\pi}{2}}), \qquad \mbox{or} \qquad
(\sigma_{\min} )_x =-\frac{1}{2} (\delta_{\bar s-\frac{\pi}{2}} + \delta_{\bar s + \frac{\pi}{2}}) 
\end{equation*}
for some $\bar s \in \R/2\pi\Z$.
\item for $\nu_{\min}$-a.e. $x \in J$ let $m^+$ and $m^-$ the traces at $x$ as in Theorem \ref{T_structure} and let $\beta\in (0,\pi)$ and 
$\bar s\in \R/2\pi\Z$ be uniquely determined by
\begin{equation*}
m^+= e^{i(\bar s +\beta)}, \qquad \mbox{and} \qquad m^-= e^{i(\bar s -\beta)}.
\end{equation*}
Then
\begin{equation*}
(\sigma_{\min} )_x = \bar g_\beta(s-\bar s)\L^1,
\end{equation*}
where $\bar g_\beta:\R/2\pi\Z \to \R$ is $\pi$-periodic and for every $s \in [0,\pi]$ is defined by
\begin{equation}\label{E_def_bargbeta}
\bar g_\beta(s):=
\begin{cases}
c(\beta) \left[ (\sin s - \cos \beta)\1_{[\pi/2-\beta,\pi/2+\beta]}(s)  \right] &  \mbox{if }\beta \in (0,\pi/4] \\
c(\beta) \left[ (\sin s - \cos \beta)\1_{[\pi/2-\beta,\pi/2+\beta]}(s) + \cos \beta - \frac{\sqrt 2}{2}\right] & \mbox{if }\beta \in (\pi/4,\pi/2] \\
\bar g_{\pi-\beta}(s) & \mbox{if }\beta \in (\pi/2,\pi),
\end{cases}
\end{equation}
and where $c(\beta) >0$ is such that 
\begin{equation*}
\int_0^{2\pi}\left|\bar g_\beta (s)\right| ds = 1.
\end{equation*}
\end{enumerate}
\end{corollary}
\begin{proof}
In particular let $\sigma$ be a minimal kinetic measure;
since $\sigma$ is $\pi$-periodic in the variable $s$, it follows from Proposition \ref{P_partial} that for $\nu_{\min}$-a.e. 
$x \in \Omega\setminus J$ it holds
\begin{equation*}
\sigma_x =\frac{1}{2+2\pi c} (\delta_{\bar s-\frac{\pi}{2}} + \delta_{\bar s + \frac{\pi}{2}} + c \L^1), \qquad \mbox{or} \qquad
\sigma_x =-\frac{1}{2+ 2\pi c} (\delta_{\bar s-\frac{\pi}{2}} + \delta_{\bar s + \frac{\pi}{2}} +c \L^1)
\end{equation*}
for some $\bar s \in \R/2\pi\Z$ and some $c \in \R$ depending on $x$. The necessary and sufficient condition \eqref{E_min} for minimality
trivially implies $c=0$.
By Theorem \ref{T_structure} and \eqref{E_kin_entropy} it holds
\begin{equation*}
n\cdot \left( \Phi(m^+)-\Phi(m^-) \right) \H^1\llcorner J = (p_x)_\sharp \left( - \partial_s\psi_\Phi \sigma \llcorner J\times \R/2\pi\Z\right).
\end{equation*}
The following identity was obtained in Section 4.2 of \cite{GL_eikonal}: for every $\beta \in [0,\pi/2]$ it holds
\begin{equation}\label{E_computation_GL}
e_1 \cdot (\Phi(e^{i\beta})-\Phi(e^{-i\beta})) = - \int_0^{2\pi}g_\beta(s) \partial_s\psi_\Phi(s)ds,
\end{equation}
where $g_\beta:\R/2\pi\Z \to \R$ is a $\pi$-periodic defined by
\begin{equation*}
g_\beta(s) = (\sin s - \cos \beta)\1_{[\pi/2-\beta,\pi/2+\beta]}(s) - \frac{2}{\pi}(\sin \beta - \beta \cos \beta) \qquad \forall s \in [0,\pi].
\end{equation*}
Let us assume first that $\beta \in [0,\pi/2]$; then with the notation in the statement we have $n= e^{i\bar s}$, therefore choosing $\tilde \Phi$ such that 
$\psi_{\tilde \Phi}(s) = \psi_\Phi(s + \bar s)$ we deduce from \eqref{E_computation_GL} that
\begin{align*}
n\cdot \left( \Phi(m^+)-\Phi(m^-) \right)  = &~ e^{i\bar s}\cdot \left( \Phi\left(e^{i(\bar s + \beta) }\right) -\Phi\left(e^{i(\bar s - \beta) }\right) \right)\\
= &~ e^{i\bar s}\cdot \int_{-\beta}^\beta \psi_{\Phi}\left(s+\bar s+ \frac{\pi}{2}\right) e^{i \left( s + \bar s +\frac{\pi}{2}\right)}ds \\
=&~ e_1 \cdot  \int_{-\beta}^{\beta} \psi_{\Phi} \left( s+\bar s+\frac{\pi}{2}\right) e^{i \left( s  +\frac{\pi}{2}\right)}ds \\
=&~ e_1 \cdot  \left( \tilde \Phi\left(e^{i\beta}\right)-\tilde \Phi\left(e^{-i\beta}\right) \right) \\
= &~ - \int_0^{2\pi}g_\beta(s)\psi'_{\tilde \Phi}(s)ds \\
= &~ - \int_0^{2\pi}g_\beta(s-\bar s)\psi'_{\Phi}(s)ds.
\end{align*}

This shows that for $\nu_{\min}$-a.e. $x \in J$ with $\beta \in (0,\pi/2)$ there exist two constants $c_1>0$ and $c_2 \in \R$ such that 
$\sigma_x= c_1(g_\beta(\cdot-\bar s) + c_2) \L^1$. It is a straightforward computation to check that the choice in \eqref{E_def_bargbeta} is the unique that satisfies the constraint in \eqref{E_min}. In particular $\sigma_x$ is uniquely determined for $\nu_{\min}$-a.e. $x \in J$ such that $\beta \in (0,\pi/2)$.

If instead $\beta \in (\pi/2,\pi)$, then $n= -e^{i\bar s}$ and it can be reduced to the previous case exchanging $m^+$ with $m^-$, 
and therefore changing the sign of $n$. Since $\partial_s \psi_\Phi$ and $g_\beta$ for $\beta\in (0,\pi/2]$ are $\pi$-periodic, then the same computations as above leads to
\begin{equation*}
n\cdot \left( \Phi(m^+)-\Phi(m^-) \right)  =  -\int_0^{2\pi}g_{\pi -\beta}(s-\bar s)\partial_s\psi_{\Phi}(s)ds.
\end{equation*}
Similarly the choice in \eqref{E_def_bargbeta} is the unique that satisfies the constraint \eqref{E_min}.
Being $\sigma_x$ uniquely determined for $\nu_{\min}$-a.e. $x \in \Omega$, the measure $\sigma_{\min}$ is unique.
\end{proof}

The following lemma links the jump set of the characteristic curves with the jump set of $m \in A(\Omega)$.
\begin{lemma}\label{L_jumps}
Let $m \in A(\Omega)$ and $\Omega'$ be a $W^{2,\infty}$ open set compactly contained in $\Omega$. Let moreover $\omega$ be a minimal Lagrangian
representation of $m$ on $\Omega'$. Then for $\omega$-a.e. $(\gamma,t^-_\gamma,t^+_\gamma) \in \Gamma$ the following property holds:
for every $t \in (t^-_\gamma, t^+_\gamma)$ such that $\gamma_s(t+)\ne \gamma_s(t-)$ it holds $\gamma_x(t)\in J$.
\end{lemma}
\begin{proof}
Since $\omega$ is a minimal Lagrangian representation, by Proposition \ref{P_Lagr} and Corollary \ref{C_disintegration} it holds
\begin{equation*}
\int_\Gamma |\sigma_\gamma| d\omega = |\sigma_\omega| = \mathcal L^1 \times |\sigma_{\min}| = \mathcal L^1 \times \left( \nu_{\min}\otimes |(\sigma_{\min})_x|\right)
\end{equation*}
as measures in $(0,T)\times \Omega'\times \R/2\pi\Z$. By Corollary \ref{C_disintegration} it follows that for 
$\mathcal L^1\times \nu_{\min}$-a.e. $(t,x) \in (0,T)\times (\Omega \setminus J)$ it holds 
\begin{equation}\label{E_supp_two_points}
\supp \left( \mathcal L^1 \times |\sigma_{\min}|\right)_{t,x} \subset \{ \bar s, \bar s + \pi\}
\end{equation}
for some $\bar s \in \R/2\pi\Z$. Suppose by contradiction that there exists
$ G \subset \Gamma$ with $\omega(G)>0$ and a measurable function $\tilde t : G \to (0,T)$ such that for every $(\gamma,t^-_\gamma,t^+_\gamma)$ in $G$ 
it holds 
\begin{equation*}
\tilde t (\gamma) \in (t^-_\gamma,t^+_\gamma), \qquad  \gamma_s\left(\tilde t(\gamma)+\right)\ne  \gamma_s\left(\tilde t(\gamma)+\right), \qquad \mbox{and} \qquad \gamma_x\left(\tilde t(\gamma)\right) \in \Omega'\setminus J.
\end{equation*}
For every $(\gamma,t^-_\gamma,t^+_\gamma) \in G$ we set 
\begin{equation*}
\tilde \sigma_\gamma = \mathcal H^1\llcorner E^+_\gamma\left(\tilde t(\gamma)\right) - \mathcal H^1\llcorner E^-_\gamma\left(\tilde t(\gamma)\right), \qquad \mbox{where} \qquad  E^\pm_\gamma\left(\tilde t(\gamma)\right) := \{ (t,x,s) \in  E^\pm_\gamma: t = \tilde t(\gamma)\}
\end{equation*}
and $E^{\pm}_\gamma$ are defined in \eqref{E_def_Epmgamma}.
Let $\tilde \sigma_\omega := \int_\Gamma |\tilde \sigma_\gamma|d\omega \in \M^+((0,T)\times \Omega'\times \R/2\pi\Z)$; 
by definition we have $\tilde \sigma_\omega\le |\sigma_\omega|$. 
Let us denote by $\tilde \nu:=(p_{t,x})_\sharp \tilde \sigma_\gamma$.
Then by definition of $\tilde \sigma_\omega$ we have that $\tilde \nu$ is concentrated on $(0,T)\times \Omega'\setminus J$ and for 
$\tilde \nu$-a.e. $(t,x) \in (0,T)\times \Omega'\setminus J$ there exist no $\bar s \in \R/2\pi\Z$ such that
$\supp (\tilde \sigma_\omega)_{t,x} \subset \{\bar s,\bar s + \pi\}$. Since $\tilde \nu(\Omega')>0$, this is in contradiction with \eqref{E_supp_two_points}.
\end{proof}

\subsection{Solutions with a single vanishing entropy}
The goal of this section is to prove the following result about solutions with vanishing entropy production.

\begin{proposition}\label{P_vanishing}
Let $\Omega\subset \R^2$ be an open set and $m\in A(\Omega)$ be such that $\div\Sigma_{\e_1,\e_2}(m)=0$. Then $J$ is contained in the union of countably many horizontal and vertical segments. Moreover $\nu_{\min}$ is concentrated on $J$.  
\end{proposition}

The result follows from Proposition \ref{P_partial} and the following general result about BV functions for which we refer to \cite{AFP_book}.

\begin{lemma}\label{L_BV_Lusin}
Let $f \in \BV((0,T);\R)$ be continuous from the right. Then for every $E\subset \R$ at most countable it holds
\begin{equation*}
\big|\tilde Df\big|\left(f^{-1}(E)\right)=0.
\end{equation*}
\end{lemma}

\begin{proof}[Proof of Proposition \ref{P_vanishing}]
We recall from \cite{GL_eikonal} that 
\begin{equation*}
\div\Sigma_{\e_1,\e_2}(m)= - 2(p_x)_\sharp \left[\sin(2s)\sigma\right].
\end{equation*} 
For $\nu_{\min}$-a.e. $x \in J$ it holds $n=\pm e^{i\bar s}$, therefore in order to show that $J$ is contained in a countable union of horizontal and vertical segments, it is sufficient to observe that for every 
$\beta\in(0,\pi)$ it holds
\begin{equation}\label{E_vanishing}
\int_{\R/2\pi\Z}g_{\beta}(s-\bar s)\sin(2s) ds = 0 \qquad \Longrightarrow \qquad \bar s \in \frac{\pi}{2}\Z.
\end{equation} 
This follows straightforwardly by the explicit expression of $g_\beta$ in \eqref{E_def_bargbeta}.
Now we prove that $\nu_{\min}$ is concentrated on $J$: by Corollary \ref{C_disintegration} and \eqref{E_vanishing}, for $\nu_{\min}$-a.e. 
$x \in \Omega\setminus J$ it holds
\begin{equation*}
\int_{\R/2\pi\Z}\sin(2s) d( \delta_{\bar s} + \delta_{\bar s + \pi})=0,
\end{equation*}
which trivially implies $\bar s \in \frac{\pi}{2}\Z$. By Lemma \ref{L_jumps}, we have 
\begin{equation*}
\L^1\times \nu_{\min}\llcorner (\Omega\setminus J) \le \int_\Gamma (\gamma_x)_\sharp |\tilde D_t \gamma_s|\left(\gamma_s^{-1}\left(\frac{\pi}{2}\Z\right)\right)d\omega(\gamma) = 0,
\end{equation*}
where in the last equality we used Lemma \ref{L_BV_Lusin}.
\end{proof}
\begin{remark}
The same argument shows that the assumption $\div\,\Sigma_{\e_1,\e_2}(m)=0$ can be replaced with $\div \Phi(m)=0$ for any $\Phi \in \mathcal E_\pi$ such that
$\{s:\partial_s \psi_\Phi(s)=0\}$ is at most countable.
\end{remark}

\section{Uniqueness of minimizers on ellipses}
The goal of this section is to prove Theorem \ref{T_unique}. Since the functional $\tilde F_0$ is invariant by rotations, then
we will assume without loss of generality that the major axis of the ellipse is parallel to $x$-axis in the plane.

The following result is essentially contained in \cite{JK_entropies} (see also \cite{IM_eikonal}): we report here the proof for completeness.
\begin{proposition}
Let $\bar u^\delta$ be defined as in Theorem \ref{T_unique}. Then $\bar u^\delta$ is a minimizer of $\tilde F_0(\cdot, \Omega_\delta)$ 
in $\Lambda^0_\delta(\Omega)$. Moreover, for every minimizer $u^\delta$ of $\tilde F_0(\cdot, \Omega_\delta)$ 
in $\Lambda^0_\delta(\Omega)$ the function $m = \nabla^\perp u^\delta$ satisfies
\begin{equation*}
\div \Sigma_{\e_1,\e_2}(m)= 0 \qquad \mbox{and} \qquad \div \Sigma_{e_1,e_2}(m)\ge 0 \qquad \mbox{in }\mathcal D'(\Omega_\delta).
\end{equation*}
\end{proposition}
\begin{proof}
In \cite{ADLM_eikonal}, the authors noticed that for every $u \in A(\Omega_\delta)$ it holds
\begin{equation}\label{E_ADM_algebraic}
\tilde F_0(u,\Omega_\delta) = \left( \left( \left|\div \Sigma_{e_1,e_2}(\nabla^\perp u)\right| (\Omega_\delta)\right)^2 + \left(\left|\div \Sigma_{\e_1,\e_2}(\nabla^\perp u)\right| (\Omega_\delta)\right)^2  \right)^{\frac{1}{2}}.
\end{equation}
Let us denote by $\bar m := \nabla^\perp \bar u^\delta$. Since for every $u \in \Lambda_\delta(\Omega)$ it holds $\nabla^\perp u = \bar m$ in 
$S_\delta$, then it follows from \eqref{E_ADM_algebraic} that
\begin{equation}\label{E_chain_minimality}
\begin{split}
\tilde F_0(u,\Omega_\delta) = &~ \left(\left( \left| \div \Sigma_{\e_1,\e_2}(\nabla^\perp u)\right|(\Omega_\delta)\right)^2 +  \left( \left| \div \Sigma_{e_1,e_2}(\nabla^\perp u)\right|(\Omega_\delta)\right)^2\right)^{\frac{1}{2}} \\
\ge &~ \div \Sigma_{e_1,e_2}(\nabla^\perp u)(\Omega_\delta) \\
= &~ \int_{\partial \Omega_\delta} \Sigma_{e_1,e_2}(\nabla^\perp u) \cdot n d\mathcal H^1 \\
= &~ \div  \Sigma_{e_1,e_2}(\bar m)(\Omega_\delta) \\
= &~ \tilde F_0(\bar u^\delta,\Omega_\delta),
\end{split}
\end{equation}
where in the last equality we used $\div \Sigma_{\e_1,\e_2}(\bar m)= 0$ and $\div \Sigma_{e_1,e_2}(\bar m)\ge 0$.
This shows in particular that $\bar u^\delta$ is a minimizer of $\tilde F_0(\cdot,\Omega_\delta)$ in $\Lambda^0_\delta(\Omega)$.
Moreover for every minimizer  $u$ of $\tilde F_0(\cdot,\Omega_\delta)$ in $\Lambda^0_\delta(\Omega)$, the inequality in 
\eqref{E_chain_minimality} is an equality and this completes the proof.
\end{proof}

\begin{theorem}
Let $\Omega$ be an ellipse, and $m \in A_\delta(\Omega)$ be such that
\begin{equation}\label{E_sign_minimality}
\div \Sigma_{\e_1,\e_2}(m)= 0, \qquad \mbox{and} \qquad \div \Sigma_{e_1,e_2}(m)\ge 0.
\end{equation}
Then
\begin{equation}\label{E_viscous}
m\llcorner \Omega = \nabla^\perp \dist(\cdot, \partial \Omega).
\end{equation}
\end{theorem}
\begin{proof}
The proof is divided into three steps: in Step 1 we link the assumptions in \eqref{E_sign_minimality} with the sign of $\partial_s \sigma_{\min}$
relying on Corollary \ref{C_disintegration} and Proposition \ref{P_vanishing}. Then we will prove in Step 2 that the entropy defect measures of
every $m$ as in the statement are concentrated on the axis of the ellipse. We finally prove in Step 3 that this last condition forces $m$ to satisfy \eqref{E_viscous}.

\emph{Step 1}. Let $m \in A_\delta(\Omega)$ be as in the statement and $\sigma_{\min}$ be its minimal kinetic measure. 
Then for every $\phi\in C^1_c(\Omega_\delta \times \R/2\pi\Z)$ such that
$\phi \ge 0$ and
\begin{equation*}
\supp \phi \subset \Omega_\delta \times \left( \left(0,\frac{\pi}{2}\right) \cup \left(\pi, \frac{3}{2}\pi\right)\right)
\end{equation*}
it holds
\begin{equation*}
\langle \partial_s\sigma_{\min}, \phi\rangle = - \int_{\Omega\times \R/2\pi\Z} \partial_s\phi d\sigma_{\min} \ge 0.
\end{equation*}
\emph{Proof of Step 1}. Since $\div \Sigma_{\e_1,\e_2}(m)= 0$, it follows from Proposition \ref{P_vanishing} 
that for $\nu_{\min}$-a.e. $x \in J$ the normal to $J$ at $x$ is
$n(x)=e^{is(x)}$ for some $s(x) \in \frac{\pi}{2}\Z$. Up to exchange $m^+$ and $m^-$, we can therefore assume without loss of generality that $n (x) = (1,0)$ or $n(x)=(0,1)$ for $\nu_{\min}$-a.e. $x \in J$. We denote by $J_h\subset J$ the points for which $n(x)=(0,1)$ and $J_v \subset J$ the points with $n(x)=(1,0)$. We consider these two cases separately:

\noindent if $n(x)=(0,1)$ then 
\begin{equation*}
\left(\div \Sigma_{e_1,e_2}(m)\right)\llcorner J_h= \frac{1}{3} \left((m^+_1)^3(x) - (m^-_1)^3(x) \right) \H^1\llcorner J_h,
\end{equation*}
therefore $m^+_1(x)=-m^-_1(x)>0$ for $\nu$-a.e. $x \in J_h$. In particular, using the same notation as in Corollary \ref{C_disintegration},
we have
$\bar s = \pi$. We observe that by the definition of $\bar g_\beta$ in \eqref{E_def_bargbeta}, for every $\beta \in (0,\pi)$ it holds 
$\partial_s \bar g_\beta(s) \ge 0$ for $\mathcal L^1$- a.e. $s \in \left(0,\frac{\pi}{2}\right) \cup \left(\pi, \frac{3}{2}\pi\right)$ and  
$\partial_s \bar g_\beta(s) \le 0$ for $\mathcal L^1$- a.e. $s \in \left(\frac{\pi}{2},\pi\right) \cup \left(\frac{3}{2}\pi,2\pi\right)$.
In particular for every $\beta \in (0,\pi)$ and $\mathcal L^1$- a.e. $s \in \left(0,\frac{\pi}{2}\right) \cup \left(\pi, \frac{3}{2}\pi\right)$ it holds
\begin{equation*}
\partial_s \bar g_{\beta}(s-\bar s) \ge 0.
\end{equation*}

\noindent Similarly if $n=(1,0)$ then 
\begin{equation*}
\left(\div \Sigma_{e_1,e_2}(m)\right)\llcorner J_v= \frac{1}{3} \left((m^+_2)^3(x) - (m^-_2)^3(x) \right) \H^1\llcorner J_v,
\end{equation*}
therefore $m^+_2(x)=-m^-_2(x)>0$ for $\nu_{\min}$-a.e. $x \in J_v$. In particular $\bar s = 0$ so that
for every $\beta \in (0,\pi)$ and $\mathcal L^1$- a.e. $s \in \left(0,\frac{\pi}{2}\right) \cup \left(\pi, \frac{3}{2}\pi\right)$ it holds
\begin{equation*}
\partial_s \bar g_{\beta}(s-\bar s) \ge 0.
\end{equation*}
Therefore by Corollary \ref{C_disintegration}, it follows 
\begin{equation*}
\langle \partial_s\sigma_{\min}, \phi\rangle = \int_\Omega \int_0^{2\pi}\phi \bar g_\beta'(s-\bar s) ds d\nu_{\min} \ge 0.
\end{equation*}
\newline \emph{Step 2}. We prove that $\nu_{\min}$ is concentrated on the axis of the ellipse.

Let us denote by
\begin{equation*}
\Omega = \left\{ x \in \R^2 : x_1^2 + a x_2^2 < r^2\right\}
\end{equation*}
with $r>0$ and $a \ge 1$.
Let us assume by contradiction that $\nu_{\min} ( J_h \cap \{x \in \R^2 : x_2>0\})>0$. Then there exists $b>0$ such that $\nu_{\min} (\{x \in J_h : x_2=b\})>0$.
By the analysis in the proof of Step 1 there exists $A \subset \R$ such that $\L^1(A)>0$ and  for $\H^1$-a.e. $x \in A \times \{b\}$ it holds $m^-_1(x)<0$. In particular we can choose $\alpha \in (\pi, 3\pi/2)$ such that 
\begin{equation}\label{E_choice_alpha}
|\tan \alpha| \le  \frac{b}{2(r+\delta)} \qquad \mbox{and} \qquad \eta:= \H^1 \left( \left\{ x \in \Omega \cap J_h : x_2=b \mbox{ and }e^{i\alpha}\cdot m^-(x)>0\right\} \right) >0.
\end{equation}
Let $\bar x_1>0$ be such that $(\bar x_1,b)\in \partial \Omega_\delta$ and denote by
\begin{equation}\label{E_def_E}
E:=\{x \in \Omega_\delta: x_2 \in (g(x_1),b)\},
\end{equation}
where $g(x_1)= \tan (\alpha) (x_1-\bar x_1) + b $. The first constraint in \eqref{E_choice_alpha} implies that $E\subset \{ x_2>0\}$ (see Figure \ref{F_ellipse1}).

\begin{figure}
\centering
\def\svgwidth{0.6\columnwidth}
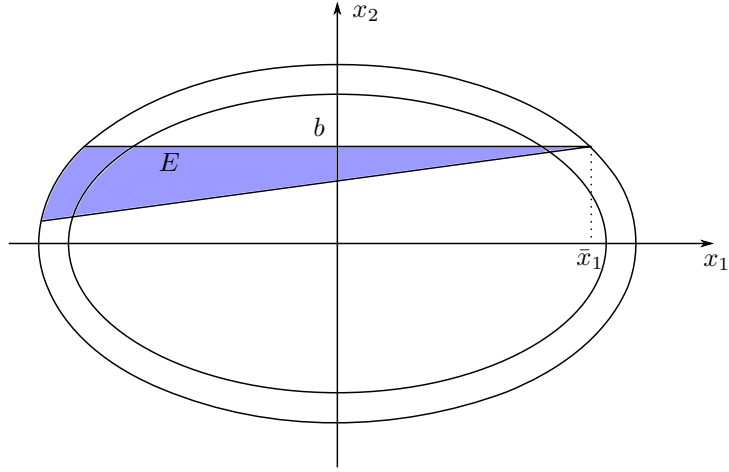
\caption{The figure illustrates the definition of $E$ in \eqref{E_def_E}.}\label{F_ellipse1}
\end{figure}

We consider the following Lipschitz approximation of the characteristic function of $E$: 
\begin{equation*}
\psi_\e ( x ) =
\begin{cases}
0 & \mbox{if }x \notin E \\
\min \left\{ 1, \frac{1}{\e}\dist(x,\partial E) \right\} & \mbox{if }x \in E.
\end{cases}
\end{equation*}
We moreover consider $\rho \in C^\infty_c(\pi + \frac{\alpha-\pi}{2},\alpha)$ such that $\rho\ge 0$ and $\int_\R \rho(s) ds =1$ and we test \eqref{E_kinetic} with $\varphi_\e (s,x) = \psi_\e(x) \rho(s)$.
If $\e < \delta$, then the choice of $\alpha$ in \eqref{E_choice_alpha} and of $\rho$ implies that
\begin{equation*}
\{(x,s) \in \Omega_\delta \times \supp(\rho) :e^{is}\cdot \nabla_x \psi_\e < 0 \} \subset \{(x,s) \in (\Omega_\delta \setminus \Omega) \times  \supp(\rho) : x_2>0 \mbox{ and }x_1<0\}.
\end{equation*}
Since $m= \bar m$ on $\Omega_\delta \setminus \Omega$, then for $\L^2 \times \L^1$-a.e. $(x,s) \in (\Omega_\delta \setminus \Omega) \times  \supp(\rho) $ it holds $\chi(x,s)=\1_{e^{is}\cdot m(x)>0}=0$. In particular, by the second condition in \eqref{E_choice_alpha} we have
\begin{equation*}\label{E_sign+}
\begin{split}
\liminf_{\e\to 0}\int_{\Omega \times \R/2\pi\Z} e^{is}\cdot \nabla_x \psi_\e(x) \rho(s) \chi(x,s) ds dx \ge &~ \int_{\{x \in \Omega:x_2=b\}\times \R/2\pi\Z} (-\sin s) \rho (s) 
\1_{e^{is} \cdot m^-(x)>0} (x) ds d\mathcal H^1(x) \\
\ge &~ \eta \sin\left(\frac{\alpha-\pi}{2}\right) \\
>&~ 0.
\end{split}
\end{equation*}
This contradicts Step 1, which implies that
\begin{equation*}
\int_{\Omega \times \R/2\pi\Z} e^{is}\cdot \nabla_x \psi_\e(x) \rho(s) \chi(x,s)ds dx = - \langle \partial_s \sigma_{\min}, \rho \otimes \psi_\e \rangle \le 0.
\end{equation*}
A similar argument excludes that $\nu_{\min}(\{x\in J_h: x_2=b\})>0$ if $b<0$ and that $\nu_{\min}(\{x\in J_v: x_1=a\})>0$ if $a\ne 0$.
See Figure \ref{F_3ellipses} which illustrates the sets $E$ that need to be considered in these cases.
\begin{figure}
\centering
\def\svgwidth{\columnwidth}
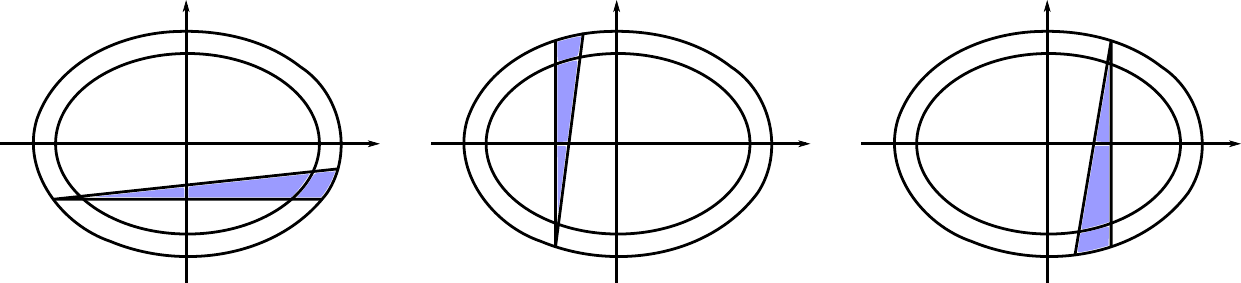
\caption{The regions in blue indicate the sets $E$ to be considered in order to repeat the presented argument in the three cases not addressed in details.}\label{F_3ellipses}
\end{figure}
\newline \emph{Step 3}. We prove that the unique $m\in A_\delta(\Omega)$ for which $\nu_{\min}$ is concentrated on the axis of the ellipse satisfies \eqref{E_viscous}. In particular we show that $m = \bar m$ on 
\begin{equation*}
\tilde \Omega_\delta = \{ x \in \Omega_\delta : x_1<0, x_2>0\},
\end{equation*}
being the argument for the other quadrants analogous. 

Let $\bar x \in \Omega$ be a Lebesgue point of $m$ and let $\bar s(\bar x) \in (\pi/2,\pi)$ be such that
\begin{equation*}
e^{i\bar s(\bar x)}= - \nabla \dist (\bar x, \partial \Omega).
\end{equation*}
For every $s \in (\pi/2,\pi)$ let $t_s>0$ be the unique value such that 
\begin{equation*}
y_s:= \bar x + t_s e^{is} \in \partial \Omega_{\delta/2}\cap \tilde \Omega_\delta.
\end{equation*}
\begin{figure}
\centering
\def\svgwidth{0.8\columnwidth}
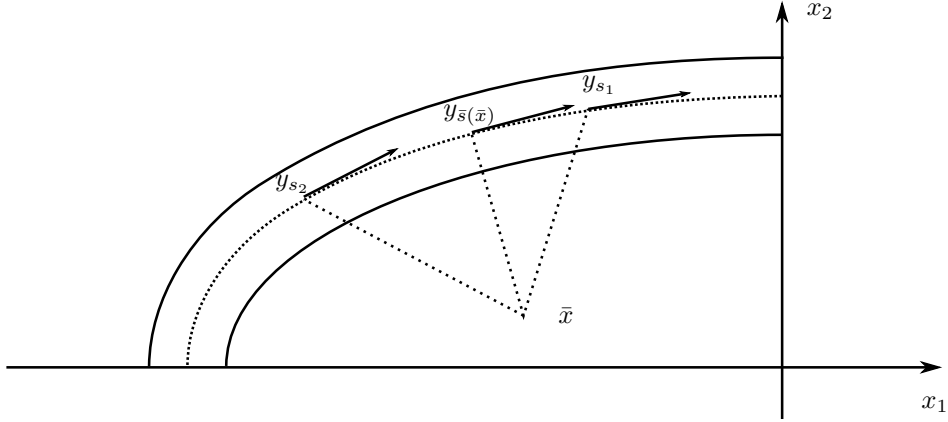
\caption{The picture represents the points $y_{s_1},y_{\bar s(\bar x)}, y_{s_2}$, while the arrows represent the values of $\bar m$ at these points.}\label{F_ellipse2}
\end{figure}
By elementary geometric considerations (see Figure \ref{F_ellipse2}) the following properties hold:
\begin{enumerate}
\item $\bar m (y_s)\cdot e^{is} >0$ for every $s \in (\pi/2,\bar s (\bar x))$;
\item $\bar m (y_s)\cdot e^{is} <0$ for every $s \in (\bar s (\bar x),\pi)$.
\end{enumerate}
In particular for every $\e \in \left(0,  \frac{1}{2}\min \{\bar s(\bar x) -\pi/2, \pi - \bar s(\bar x)\}\right)$ there exists $r \in (0,\frac{\delta}{2})$ such that
\begin{enumerate}
\item for every $s \in (\bar s(\bar x) - 2 \e, \bar s (\bar x) - \e)$ and every $y \in B_r(y_s)$ it holds $\bar m (y) \cdot e^{is}>0$;
\item for every $s \in (\bar s(\bar x) + \e, \bar s (\bar x) +2 \e)$ and every $y \in B_r(y_s)$ it holds $\bar m (y) \cdot e^{is}<0$.
\end{enumerate}
By Step 2 we have that
\begin{equation*}
e^{is}\cdot \nabla_x \chi = 0 \qquad \mbox{in }\D'(\tilde \Omega_\delta)
\end{equation*}
therefore for $\mathcal L^1$-a.e. $s \in \R/2\pi\Z$ the sets $\left\{x \in \tilde \Omega_\delta : e^{is}\cdot m(x)>0\right\}$ and 
$\left\{x \in \tilde \Omega_\delta : e^{is}\cdot m(x)<0\right\}$ are invariant
by translations in the direction $e^{is}$ up to negligible sets. Since $m = \bar m$ in $\tilde \Omega_\delta \setminus \Omega$, then
it follows by the previous analysis that for every $\e>0$ there exists $r>0$ such that for $\mathcal L^2$-a.e. $x \in B_r(\bar x)$
the following two inequalities hold:
\begin{equation}\label{E_two_conditions}
\begin{split}
m(x)\cdot e^{is}>0 &\qquad \mbox{for }\mathcal L^1\mbox{-.a.e. }s \in  (\bar s(\bar x) - 2 \e, \bar s (\bar x) - \e), \\
m(x)\cdot e^{is}<0 &\qquad \mbox{for }\mathcal L^1\mbox{-.a.e. }s \in  (\bar s(\bar x) +\e , \bar s (\bar x) +2 \e).
\end{split}
\end{equation}
The two conditions in \eqref{E_two_conditions} implies that for $\mathcal L^2$-a.e. $x \in B_r(\bar x)$ it holds
$m(x)=e^{is(x)}$ for some $s(x) \in [\bar s(\bar x)- \pi/2 -\e,  \bar s(\bar x)- \pi/2 +\e]$.
Since $\bar x$ is a Lebesgue point of $m$, letting $\e \to 0$ we obtain 
\begin{equation*}
m(\bar x)= \bar s(\bar x) - \frac{\pi}{2} = \bar m(\bar x).
\end{equation*}
This concludes the proof.
\end{proof}

\bibliographystyle{alpha}

\begin{thebibliography}{DMKO01}

\bibitem[AC14]{AC_superposition}
Luigi Ambrosio and Gianluca Crippa.
\newblock Continuity equations and {ODE} flows with non-smooth velocity.
\newblock {\em Proc. Roy. Soc. Edinburgh Sect. A}, 144(6):1191--1244, 2014.

\bibitem[ADLM99]{ADLM_eikonal}
Luigi Ambrosio, Camillo De~Lellis, and Carlo Mantegazza.
\newblock Line energies for gradient vector fields in the plane.
\newblock {\em Calc. Var. Partial Differential Equations}, 9(4):327--255, 1999.

\bibitem[AFP00]{AFP_book}
L.~Ambrosio, N.~Fusco, and D.~Pallara.
\newblock {\em Functions of Bounded Variation and Free Discontinuity Problems}.
\newblock Oxford Science Publications. Clarendon Press, 2000.

\bibitem[AG87]{AG_conjecture}
Patricio Aviles and Yoshikazu Giga.
\newblock A mathematical problem related to the physical theory of liquid
  crystal configurations.
\newblock In {\em Miniconference on geometry and partial differential
  equations, 2 ({C}anberra, 1986)}, volume~12 of {\em Proc. Centre Math. Anal.
  Austral. Nat. Univ.}, pages 1--16. Austral. Nat. Univ., Canberra, 1987.

\bibitem[AG96]{AG_viscosity}
Patricio Aviles and Yoshikazu Giga.
\newblock The distance function and defect energy.
\newblock {\em Proc. Roy. Soc. Edinburgh Sect. A}, 126(5):923--938, 1996.

\bibitem[AG99]{AG_semicontinuity}
Patricio Aviles and Yoshikazu Giga.
\newblock On lower semicontinuity of a defect energy obtained by a singular
  limit of the {G}inzburg-{L}andau type energy for gradient fields.
\newblock {\em Proc. Roy. Soc. Edinburgh Sect. A}, 129(1):1--17, 1999.

\bibitem[AKLR02]{AKLR_rectifiability}
Luigi Ambrosio, Bernd Kirchheim, Myriam Lecumberry, and Tristan Rivi{\`e}re.
\newblock On the rectifiability of defect measures arising in a micromagnetics
  model.
\newblock In {\em Nonlinear problems in mathematical physics and related
  topics, {II}}, volume~2 of {\em Int. Math. Ser. (N. Y.)}, pages 29--60.
  Kluwer/Plenum, New York, 2002.

\bibitem[ALR03]{ALR_viscosity}
Luigi Ambrosio, Myriam Lecumberry, and Tristan Rivi{\`e}re.
\newblock A viscosity property of minimizing micromagnetic configurations.
\newblock {\em Comm. Pure Appl. Math.}, 56(6):681--688, 2003.

\bibitem[BBM17]{BBM_multid}
S.~Bianchini, P.~Bonicatto, and E.~Marconi.
\newblock A lagrangian approach to multidimensional conservation laws.
\newblock {\em preprint SISSA 36/MATE}, 2017.

\bibitem[CDL07]{CDL_limsup}
Sergio Conti and Camillo De~Lellis.
\newblock Sharp upper bounds for a variational problem with singular
  perturbation.
\newblock {\em Math. Ann.}, 338(1):119--146, 2007.

\bibitem[DLO03]{DLO_JEMS}
Camillo De~Lellis and Felix Otto.
\newblock Structure of entropy solutions to the eikonal equation.
\newblock {\em J. Eur. Math. Soc. (JEMS)}, 5(2):107--145, 2003.

\bibitem[DMKO01]{DMKO_compactness}
Antonio DeSimone, Stefan M\"{u}ller, Robert~V. Kohn, and Felix Otto.
\newblock A compactness result in the gradient theory of phase transitions.
\newblock {\em Proc. Roy. Soc. Edinburgh Sect. A}, 131(4):833--844, 2001.

\bibitem[GL20]{GL_eikonal}
Francesco Ghiraldin and Xavier Lamy.
\newblock Optimal {B}esov differentiability for entropy solutions of the
  eikonal equation.
\newblock {\em Comm. Pure Appl. Math.}, 73(2):317--349, 2020.

\bibitem[Ign12]{Ignat_confluentes}
Radu Ignat.
\newblock Singularities of divergence-free vector fields with values into $S^1$ or $S^2$. Application to micromagnetics.
\newblock {\em Confluentes Mathematici}, 4(3):1--80, 2012.

\bibitem[IM12]{IM_eikonal}
Radu Ignat and Beno\^{\i}t Merlet.
\newblock Entropy method for line-energies.
\newblock {\em Calc. Var. Partial Differential Equations}, 44(3-4):375--418,
  2012.

\bibitem[JK00]{JK_entropies}
W.~Jin and R.~V. Kohn.
\newblock Singular perturbation and the energy of folds.
\newblock {\em J. Nonlinear Sci.}, 10(3):355--390, 2000.

\bibitem[JOP02]{JOP_zero-energy}
Pierre-Emmanuel Jabin, Felix Otto, and Beno\^{\i}t Perthame.
\newblock Line-energy {G}inzburg-{L}andau models: zero-energy states.
\newblock {\em Ann. Sc. Norm. Super. Pisa Cl. Sci. (5)}, 1(1):187--202, 2002.

\bibitem[JP01]{JP_kinetic}
Pierre-Emmanuel Jabin and Beno\^{\i}t Perthame.
\newblock Compactness in {G}inzburg-{L}andau energy by kinetic averaging.
\newblock {\em Comm. Pure Appl. Math.}, 54(9):1096--1109, 2001.

\bibitem[Lor12]{Lorent_simple}
Andrew Lorent.
\newblock A simple proof of the characterization of functions of low {A}viles
  {G}iga energy on a ball {\it via} regularity.
\newblock {\em ESAIM Control Optim. Calc. Var.}, 18(2):383--400, 2012.

\bibitem[Lor14]{Lorent_quantitative}
Andrew Lorent.
\newblock A quantitative characterisation of functions with low {A}viles {G}iga
  energy on convex domains.
\newblock {\em Ann. Sc. Norm. Super. Pisa Cl. Sci. (5)}, 13(1):1--66, 2014.

\bibitem[LP18]{LP_two_entropies}
Andrew Lorent and Guanying Peng.
\newblock Regularity of the eikonal equation with two vanishing entropies.
\newblock {\em Ann. Inst. H. Poincar\'{e} Anal. Non Lin\'{e}aire},
  35(2):481--516, 2018.

\bibitem[LP21]{LP_factorization}
Andrew Lorent and Guanying Peng.
\newblock Factorization for entropy production of the eikonal equation and
  regularity.
\newblock {\em arXiv:2104.01467v1}, 2021.

\bibitem[LPT94]{LPT_kinetic}
P.-L. Lions, B.~Perthame, and E.~Tadmor.
\newblock A kinetic formulation of multidimensional scalar conservation laws
  and related equations.
\newblock {\em J. Amer. Math. Soc.}, 7(1):169--191, 1994.

\bibitem[Mar19]{M_Lebesgue}
Elio Marconi.
\newblock On the structure of weak solutions to scalar conservation laws with
  finite entropy production.
\newblock {\em ArXiv:1909.07257}, 2019.

\bibitem[Mar20a]{M_RS}
Elio Marconi.
\newblock Rectifiability of entropy defect measures in a micromagnetics model.
\newblock {\em ArXiv:2011.13065v1}, 2020.

\bibitem[Mar20b]{M_Burgers}
Elio Marconi.
\newblock The rectifiability of the entropy defect measure for burgers
  equation.
\newblock {\em arXiv:2004.09932}, 2020.

\bibitem[OG94]{OG_morphology}
Michael Ortiz and Gustavo Gioia.
\newblock The morphology and folding patterns of buckling-driven thin-film
  blisters.
\newblock {\em J. Mech. Phys. Solids}, 42(3):531--559, 1994.

\bibitem[RS01]{RS_magnetism}
Tristan Rivi{\`e}re and Sylvia Serfaty.
\newblock Limiting domain wall energy for a problem related to micromagnetics.
\newblock {\em Comm. Pure Appl. Math.}, 54(3):294--338, 2001.

\bibitem[RS03]{RS_magnetism2}
Tristan Rivi{\`e}re and Sylvia Serfaty.
\newblock Compactness, kinetic formulation, and entropies for a problem related
  to micromagnetics.
\newblock {\em Comm. Partial Differential Equations}, 28(1-2):249--269, 2003.

\bibitem[Vas01]{Vasseur_traces}
A.~Vasseur.
\newblock Strong traces for solutions of multidimensional scalar conservation
  laws.
\newblock {\em Arch. Ration. Mech. Anal.}, 160(3):181--193, 2001.

\end{thebibliography}

\end{document}